\documentclass[reqno, a4paper, 11pt]{article}
\usepackage{amsmath, amssymb, amsthm, enumitem, mathrsfs}
\usepackage{placeins}

\pdfoutput=1
\usepackage{color}

\usepackage{natbib}

\usepackage{hyperref}
\usepackage{txfonts}

\usepackage{newtxtext}

\usepackage{tikz}
\usetikzlibrary{external}
\usetikzlibrary{backgrounds,patterns,decorations.pathreplacing,decorations.pathmorphing}
\usepackage{caption}
\usepackage{accents}

\newtheorem{theorem}{Theorem}

\newtheorem{setup}[theorem]{Setup}
\newtheorem{corollary}[theorem]{Corollary}

\newtheorem{lemma}[theorem]{Lemma}
\newtheorem{proposition}[theorem]{Proposition}

{Important Convention}

\theoremstyle{remark}
\newtheorem{remark}[theorem]{Remark}
\newtheorem{example}[theorem]{Example}

  \newcommand{\F}{\mathcal{F}}

 \newcommand{\X}{\mathcal{X}}

 \renewcommand{\phi}{\varphi}

\newcommand{\E}{\mathbb{E}}

\renewcommand{\P}{\mathbb{P}}
\newcommand{\N}{\mathbb{N}}
\newcommand{\Q}{\mathbb{Q}}
\newcommand{\R}{\mathbb{R}}

\newcommand{\bes}{\begin{subequations}}
\newcommand{\ees}{\end{subequations}}
\newcommand{\eea}{\end{eqnarray}}

\newcommand{\EE}{{\mathbb E}}

\newcommand{\QQ}{{\mathbb Q}}

\usepackage{bbm}

\renewcommand{\epsilon}{\varepsilon}

\DeclareMathOperator{\proj}{proj}

\newcommand{\fourIdx}[5]{%
\setbox1=\hbox{\ensuremath{^{#1}}}%
 \setbox2=\hbox{\ensuremath{_{#2}}}%
 \setbox5=\hbox{\ensuremath{#5}}%
 \hspace{\ifnum\wd1>\wd2\wd1\else\wd2\fi}%
 \ensuremath{\copy5^{\hspace{-\wd1}\hspace{-\wd5}#1\hspace{\wd5}#3}%
 _{\hspace{-\wd2}\hspace{-\wd5}#2\hspace{\wd5}#4}%
 }}

\numberwithin{equation}{section}
\numberwithin{theorem}{section}

\renewcommand{\subset}{\subseteq}

\usepackage{subcaption}
\usepackage{graphicx}

\renewcommand{\mathrm}{}

\makeatletter
\newcommand{\mylabel}[2]{#2\def\@currentlabel{#2}\label{#1}}
\makeatother

\newcommand{\WA}{\mathcal W}

\newcommand{\Y}{\mathcal Y}
\newcommand{\Pc}{\mathcal P}

\usepackage{relsize}
\def\fcmp{\mathbin{\raise 0.6ex\hbox{\oalign{\hfil$\scriptscriptstyle \mathrm{o}$\hfil\cr\hfil$\scriptscriptstyle\mathrm{9}$\hfil}}}}

\newcommand{\Law}{\mathscr L}

\newcommand{\tint}{{\textstyle \int}}

\begin{document}

\title
{Weak Transport for Non-Convex Costs and\\ Model-independence in a Fixed-Income Market}
\author{B.\ Acciaio\thanks{Corresponding author. ETH Zurich, \emph{beatrice.acciaio@math.ethz.ch}} \and M.\ Beiglb\"ock
\thanks{University of Vienna} \and G. Pammer\thanks{ETH Zurich}}

\maketitle

\begin{abstract}
We consider a model-independent pricing problem in a fixed-income market and show that it leads to a weak optimal transport problem as introduced by Gozlan et al.
We use this to characterize the extremal models for the pricing of caplets on the spot rate and to establish a first robust super-replication result that is applicable to fixed-income markets. 

Notably, the weak transport problem exhibits a cost function which is non-convex and thus not covered by the standard assumptions of the theory.
In an independent section, we establish that weak transport problems for general costs can be reduced to equivalent problems that do satisfy the convexity assumption, extending the scope of weak transport theory. 
This part could be of its own interest independent of our financial application, and is accessible to readers who are not familiar with mathematical finance  notation.

\medskip

\noindent\emph{keywords:}  fixed-income markets, robust pricing and hedging, weak transport problem
\end{abstract}

\section{Introduction}

Prompted by Hobson's seminal paper \citep{Ho98a}, the field  of model-independent  finance has seen a steep  development. 
Typically the framework consists in a market where some assets are dynamically traded, while some derivatives are statically traded at time zero.  
Then, without imposing any model or underlying probability measure, one seeks robust pricing bounds for other derivatives. Usually the payoffs are already expressed in discounted terms, and the pricing problem corresponds to looking for market-compatible martingale measures, that is, probability measures on the path space such that the discounted asset prices - taken to be the canonical processes - are martingales, and  reproduce the observed marked prices.

A widely used  assumption is the availability, for static trading, of call options with maturity $T$ written on an asset $S$, for all strikes $K$, which determines the distribution of $S_T$ under  the pricing measure.  However, this reasoning crucially relies on the existence of a deterministic
num\'eraire for discounting. This point is crucial for the use of tools from the Optimal Transport and Skorokhod embedding, see  \cite{Ho11}, \cite{Ob04}, \cite{BeHePe12}, \cite{GaHeTo13}, \cite{BeJu16}, \cite{DoSo12}, \cite{CaLaMa14}, \cite{BeCoHu14}, \cite{CoObTo19}, \cite{cheridito2021martingale}, among many others.

In the present paper, we do not assume the existence of a deterministic bank account, or of a bank account at all. Instead, we consider a fixed-income market, where  bonds are dynamically traded, and some options on them are statically traded. 
 A similar setting is considered
by \cite{AG20}, who assume finitely many call options with possibly different maturities being traded on a bond, and investigate market consistency with absence of arbitrage. This analysis follows the spirit of  \cite{DH07}, and displays in a simple market the different characteristics of a robust framework when stochastic discounting is allowed.
On the other hand,  the focus of the present paper is on robust pricing.
We will use bonds as num\'eraires,
thus the pricing measures will be the so-called forward measures, a notion that we recall in the next paragraph.

For $T>0$, we refer to a (zero coupon) bond with maturity $T$ as a $T$-bond, and denote its price at time $t\leq T$ by $p(t,T)$.
A $T$-forward measure $\QQ_T$ is a probability measure on $\F_T$ (where $(\Omega, \F, (\F_t)_{t\leq T})$ is some abstract filtered space) such that every traded asset expressed in units of  the $T$-bond is a martingale. The pricing formula, at time $s$, for a claim $\Phi$ with maturity $t$, for $s\leq t\leq T$, is then given by
\begin{equation}\label{eq.pr.fwd}
\tfrac{V_s(\Phi)}{p(s,T)}=\EE^{\QQ_T}\left[\tfrac{\Phi}{p(t,T)}\Big|\F_s\right].
\end{equation}
In particular, having the prices at time $0$ of all call options with maturity $T$ on a given asset, identifies the distribution of that asset at time $T$ under the $T$-forward measure.

\subsection{Robust pricing setting}
\label{sect.intro_pr}

\begin{setup}\label{Tradedcalls} We consider maturities $T_0=0<T_1<T_2<T_3$, and let $T_2$- and $T_3$-bonds be (dynamically) traded in the market.\footnote{All our results are still valid if $T_1$-bonds are also dynamically traded.}
We also let call options with maturity $T_1$ on the $T_2$-bond, and call options with maturity $T_2$ on the $T_3$-bond be (statically) traded at time zero, for every strike $K$, and denote the respective prices by $C(T_1,T_2,K)$ and $C(T_2,T_3,K)$. 
\end{setup}
A main contribution of this paper is to provide a 
robust superreplication theorem that applies to general derivatives written on $p(T_1, T_2), p(T_1, T_3)$, i.e.\ we will consider payoffs of the form 
\begin{align}\label{Ex.payoff}
\Phi(p(T_1, T_2), p(T_1, T_3)).
\end{align}
\begin{example}
The following example fits the above setup well.
Let us recall that, for $0\leq S<T$, and a simple spot rate $F(S,T)$ prevailing at $S$, a caplet with reset date $S$, settlement date $T$, and strike rate $K$, yields the following cash-flow at time $T$:
\[
(T-S)(F(S,T)-K)^+.
\]
Note that this corresponds to the following cash-flow at time $S$:
\[
(1+(T-S)K)\left(\frac{1}{1+(T-S)K}-p(S,T)\right)^+,
\]
that is, to holding $1+(T-S)K$ puts with maturity $S$ and strike $1/(1+(T-S)K)$ written on the $T$-bond; see \cite[Section~2.6.1]{Fi_book}.

We now invoke the fact that $6$ month caplets (i.e. with $T-S=6$ months) are much more liquid than the $12$ month ones. As common in robust finance literature, we stretch this fact and actually assume that $6$ month caplets are liquidly traded for every strike rate.
In particular, for $T_{i+1}-T_i=6$ months, $i=1,2$,
this means availability of all caplets with reset date $T_i$ and settlement date $T_{i+1}$, for $i=1,2$. 
From the above, this corresponds to knowing the prices of puts with maturity $T_i$ written on the $T_{i+1}$-bond for every strike, for $i=1,2$.
This is captured by Setup~\ref{Tradedcalls}, and the analysis of the current paper allows to obtain lower and upper pricing  bounds for options of the form \eqref{Ex.payoff}, thus in particular of the (less liquid) $12$ month caplets with reset date $T_1$ and settlement date $T_3$, i.e., $\Phi=(1+(T_3-T_1)K)\left(\frac{1}{1+(T_3-T_1)K}-p(T_1,T_3)\right)^+$. In fact, in this case we can also identify the extremal models, as well as the corresponding sub- and superreplication strategies. (see Section~\ref{sect.caplets}). 
\end{example}

\subsection{Robust superreplication theorem}\label{sect.intro_super}
In our setup it is most convenient to take the $T_2$-bond as a num\'eraire. In particular,  we consider the \emph{discounted} processes
\begin{align}\label{eq.disc.intro}
X_i:= \tfrac{p(T_i, T_1)}{p(T_i, T_2)}, i=0,1, \quad 
Y_i:= \tfrac{p(T_i, T_3)}{p(T_i, T_2)}, i=0,1,2.
\end{align}
For the convenience of the reader, we detail in Section~\ref{sect.disc} below how to switch between undiscounted quantities and quantities expressed in $T_2$-bond units. 
In the present introductory section we will consistently refer to quantities expressed in $T_2$-bonds.
 The assumption in Setup~\ref{Tradedcalls} then asserts that options on $X_1$ and $Y_2$ are liquidly traded. In the spirit of \cite{BL78}, this amounts to the identification of probabilities $\mu, \nu$ on $\R_+$ with finite first moments, such that for derivatives $\phi$, $\psi$
\begin{align}\label{DualCosts}
\text{price}(\phi(X_1))= \tint \phi\, d\mu, \quad 
\text{price}(\psi(Y_2))= \tint \psi\, d\nu.
\end{align}
 In view of the classical (in the sense of \emph{non-robust}) theory of fixed-income markets, we expect  pricing functionals to arise from $T_2$-forward measures. 
  A generic $T_2$-forward measure is compatible with the information given by the market if 
  (compare \eqref{eq.pr.fwd})  
\begin{align}\label{eq.mk.pr}
\text{law}_{\QQ_{T_2}}(X_1) \sim \mu, \quad \text{law}_{\QQ_{T_2}}(Y_2) \sim \nu.
\end{align}
We write $\mathcal Q(\mu, \nu)$ for the class of all setups $(\Omega, \F, (\F_{T_i})_{i=0}^2, \QQ_{T_2}), p(T_i, T_j), i\leq j \leq 3, i\leq 2$, satisfying the \emph{marginal constraints} \eqref{eq.mk.pr} and 
the \emph{martingale constraints}  
\begin{align}
(X_i)_{i=0,1},\quad  (Y_i)_{i=0,1,2} \quad \text{ are $(\F_{T_i})$-martingales}. 
\end{align} 
(Note that we make no further assumptions on the underlying stochastic basis.)
We will slightly abuse notation in that we write ${\QQ_{T_2}}\in\mathcal Q(\mu, \nu)$ when we really mean that we run over all such setups.  
Writing  $\hat \Phi(X_1, Y_2)$ for the payoff \eqref{Ex.payoff} in terms of $T_2$-bonds (see Section~\ref{sect.disc}), we arrive at the primal optimization problem 
\begin{align}\label{eq.rob}
\hat P_F  :=
\sup_{\QQ_{T_2}\in\mathcal{Q}(\mu,\nu)}
\EE^{\QQ_{T_2}}[\hat\Phi(X_1, Y_1)].
\end{align}
We stress that problem \eqref{eq.rob} differs significantly from a usual martingale optimal transport problem in that our objective functional is written on $Y_1$, while the marginal constraint concerns $Y_2$. We also note that, in contrast to most problems considered in robust finance, it is delicate to pass from \eqref{eq.rob} to a problem on a canonical setup, cf.\ Proposition~\ref{WTFTheorem}. 

Switching to the dual problem of \eqref{eq.rob}, one is allowed to trade statically in vanilla payoffs $\phi, \psi$ written on $X_1$, $Y_2$ and to trade dynamically in $(Y_i)_{i=1,2}$. Therefore we consider superhedges of the form\footnote{In principle we could allow for the hedge $\Delta$ to depend also on the value of $X_1$, but as part of our results we will obtain that this does not change the actual value of the superhedging problem.}
\begin{align}\label{Rob.sup.hed} 
\hat\Phi(X_1, Y_1) \leq \phi(X_1) +  \psi(Y_2) + \Delta(Y_1) (Y_2-Y_1). 
\end{align}
Since our goal is to obtain a robust superhedge, we shall require that \eqref{Rob.sup.hed}  holds for arbitrary $(X_1, Y_1, Y_2) = (x_1, y_1, y_2) \in \R^3_+$.  According to \eqref{DualCosts}, the costs  for this strategy (in discounted terms) amount to 
$\tint \phi\,  d\mu + \tint \psi\,  d\nu.$ 
Our main superreplication theorem is then:
\begin{theorem}\label{SuperRepIntro}
Assume that $\mu, \nu$ are probabilities on $\R_+$ with finite $r$-th moments, for some $r \geq 1$. Assume that $\hat \Phi(x,y)$ is upper semicontinuous and bounded from above by a multiple of $|x|^r + |y|^r$.
Then we have
\begin{align*}
\hat P_F= \inf\left\{\tint\phi\, d\mu+ \tint \psi \, d\nu: \hat\Phi(x_1,y_1) \leq \phi(x_1) +  \psi(y_2) + 
\Delta(y_1)(y_2-y_1)\, \mbox{ for all } x_1,y_1, y_2 \geq 0\right\},
\end{align*}
where the infimum is taken over continuous functions $\phi$ and $\psi$ that are bounded in absolute terms by a multiple of $x\mapsto (1+|x|)^r$ and $\Delta$ is an arbitrary function\footnote{In fact it suffices to consider functions $\Delta$ which are increasing.}.
\end{theorem}

\begin{remark}
It is worth noticing that the $T_3$-bond plays no special role for our results, in the sense that, by replacing it with a generic asset $S$, we would get the same superreplication result. That is, if in Setup~\ref{Tradedcalls}, instead of the $T_3$-bond, we assume that a generic asset $S$ is dynamically traded and $T_2$-calls on it statically traded, then Theorem~\ref{SuperRepIntro} holds true for any derivative of the form $\Phi(p(T_1, T_2), S_{T_1})$. In this case, from market prices we deduce the distributions under $\QQ_{T_2}$ of $X_1=1/p(T_1,T_2)$ and of $S_{T_2}$.
\end{remark}

Robust versions of the superhedging duality / FTAP have received particular attention in robust finance, see e.g.\ \cite{HoNe12}, \cite{ABPS16}, \cite{BoNu13}, \cite{BFM17}, \cite{ChKuTa17}, among many others. On the other hand, ambiguity in fixed-income markets has been recently studied by  \cite{Ho18}, \cite{FaSc19}, \cite{Ho19}, \cite{Ho20}, that assume specific models (Heath-Jarrow-Morton, Hull-White) and consider a non-dominated set of probability measures. On the other hand in the current paper, we consider a robust framework for fixed-income markets, without assuming any model nor any set of probability measures, and provide first superhedging duality results and characterization of extremal models.

\subsection{Weak transport formulation}

Our proof of Theorem \ref{SuperRepIntro} is based on new results in the theory of weak optimal transport (WOT) (see Section~\ref{sect.WOT}).
Specifically, we will show that WOT problems for general costs can be reduced to the much better understood case of WOT problems that satisfy convexity constraints. This reduction enables us to rewrite the pricing problem \eqref{eq.rob} as a weak transport problem for barycentric costs. To formulate this, we write $\Pi(\mu, \nu)$ for the set of all couplings of the probabilities $\mu, \nu$, and for $\pi \in \Pi(\mu, \nu)$ we denote its disintegration w.r.t. $\mu$ by $(\pi_x)_x$. For  a measure $p$ on $\R$ that has a finite first moment we write $b(p)$ for its barycenter.
Using these notions, it is possible to rewrite \eqref{eq.rob} as an optimization problem over measures on $\R^2$. 

\begin{proposition}\label{WTFTheorem}
Assume that $\mu, \nu$ are probabilities on $\R_+$ with finite $r$-th moment, for some $r\geq 1$, and that $\mu$ is continuous. Assume that $\hat \Phi(x,y)$ is continuous and has at most growth of order $|x|^r + |y|^r$. Then
\begin{align}\label{WTFormulation}
\hat P_F=\sup_{\pi \in \Pi(\mu, \nu)}\int \hat\Phi(x, b(\pi_x))\, \mu(dx).
\end{align}
\end{proposition}
The basic idea behind the reformulation in \eqref{WTFormulation} is to reinterpret the martingale property of $(Y_1, Y_2)$ in terms of barycenters of the disintegration of a transport plan between the prescribed marginals.

The formulation of the robust pricing problem given in Proposition \ref{WTFTheorem} plays a key role in our explicit solution of the weak transport problem in the case where $\Phi$ is a put or call option (see Section~\ref{sect.caplets}).  \eqref{WTFormulation} can also be used to formulate our superreplication theorem \ref{SuperRepIntro} in the language of optimal transport. We state here the corresponding theorem to highlight the connection with the classical Monge-Kantorovich duality:
\begin{theorem}\label{BaryDuality} 
Assume that $\mu$ is a continuous probability on a Polish space $\mathcal{X}$ and $\nu$ a probability on $\R^d$, $d\in \N$, with finite $r$-th moments, for some $r \geq 1$. Assume that $c:\mathcal{X}\times \R^d\to \R$ is continuous and has at most growth of order $d_\X(x,x_0)^r + |y|^r$ (where $x_0$ is an arbitrary fixed point in $\X$ and $d_\X$ is a compatible, complete metric on $\X$).
Then we have
\begin{align}&\hspace{-2mm} \sup_{\pi \in \Pi(\mu, \nu)}\tint c(x, b(\pi_x))\, \mu(dx) \\ = &\,  \inf\left\{\tint\phi\, d\mu+ \tint \psi \, d\nu: c(x, y_1) \leq \phi(x) +  \psi(y_2) + 
\Delta(y_1) (y_2 - y_1)\, \mbox{ for all } x,y_1, y_2\right\}  \label{Dref1} \\
= &\,  \inf\left\{\tint\phi\, d\mu+ \tint \psi \, d\nu: c(x, y) \leq \phi(x) +  \psi(y) \, \mbox{ for all } x,y, \text{ $\psi$ convex}\right\}, \label{Dref2}
\end{align}
where the infimum is taken over continuous functions $\phi, \psi$ that are bounded in absolute terms by a multiple of $x\mapsto (1+d_\X(x,x_0))^r$, $y\mapsto (1+|y|)^r$, respectively{, and $\Delta$ is measurable.}
\end{theorem}
We note that equality between \eqref{Dref1} and \eqref{Dref2} appears quite natural upon expressing $\phi$ in terms of an $\inf$/$\sup$ convolution.

\subsection{Organisation of the paper.}
In Section~\ref{sect.robust} we reformulate the robust pricing problem in discounted terms (with all quantities expressed in units of $T_2$-bonds) and illustrate in the case of caplets our results on superreplication duality and characterization of primal and dual optimizers. In Section~\ref{sect.WOT} we formulate and prove our new results for the weak optimal transport problem, and from these we derive the results stated in the introduction.

\section{Robust pricing problem}\label{sect.robust}
\subsection{On discounting. }\label{sect.disc}
Throughout this section, we work in the framework described in Sections~\ref{sect.intro_pr}-\ref{sect.intro_super}, thus in particular under Setup~\ref{Tradedcalls}. Recall that we discount our financial instruments by expressing them in terms of $T_2$-bonds, see \eqref{eq.disc.intro}. In particular we consider
\begin{align}\label{eq:T2discount}
X_1=\tfrac1{p(T_1,T_2)}, \quad Y_1=\tfrac{p(T_1,T_3)}{p(T_1,T_2)},\quad Y_2= p(T_2, T_3).\end{align}
Then a generic $T_2$-forward measure is compatible with the information given by the market if for all $K\geq 0$ (cf. \eqref{eq.pr.fwd})  
\begin{align}\label{eq.mk.prC}
\begin{split} 
\frac{C(T_1,T_2,K)}{p(0,T_2)}&=\EE^{\QQ_{T_2}}\left[\frac{\left({p(T_1,T_2)}-K\right)^+}{p(T_1,T_2)}\right]=\EE^{\QQ_{T_2}}\left[\left(1-KX_1\right)^+\right], \\
\frac{C(T_2,T_3,K)}{p(0,T_2)}&=\EE^{\QQ_{T_2}}\left[\frac{\left({p(T_2,T_3)}-K\right)^+}{p(T_2,T_2)}\right]=\EE^{\QQ_{T_2}}\left[(Y_2-K)^+\right].
\end{split}
\end{align}
Therefore, by \cite{BL78}, from the observed call prices we can deduce the distribution under $\QQ_{T_2}$ of $X_1$ and $Y_2$, which we denoted by $\mu$ and $\nu$, respectively.

We are interested in a derivative of the form $\Phi(p(T_1, T_2), p(T_1,T_3))$, with maturity $T_1$. In terms of $T_2$-bonds, its payoff equals
$$
\frac{\Phi(p(T_1, T_2), p(T_1,T_3))}{p(T_1,T_2)} 
= \Phi\left(\frac1{X_1},\frac{ Y_1}{X_1}\right)X_1.$$
This justifies our notation in Section \ref{sect.intro_super}, where we denoted the discounted payoff as a function of $X_1,Y_1$:
\begin{equation}\label{eq.phidisc}
\hat\Phi (X_1, Y_1),
\end{equation}
i.e.\ formally we define 
\begin{align}
\hat \Phi(x_1,y_1):=\Phi\left( \frac 1{x_1}, \frac{y_1}{x_1}\right)x_1.
\end{align} 
According to a $T_2$-forward measure $\Q_{T_2}$, the discounted price for such a derivative is then 
$$\EE^{\QQ_{T_2}}\left[\frac{\Phi(p(T_1, T_2), p(T_1,T_3))}{p(T_1,T_2)}\right] = 
\EE^{\QQ_{T_2}}\left[\hat\Phi (X_1, Y_1)\right]. 
$$
Therefore, the robust price bounds are given by
optimization problems (cf. \eqref{eq.rob})
\begin{eqnarray}\label{eq.rob_bounds}
\inf_{\QQ_{T_2}\in\mathcal{Q}(\mu,\nu)}
\EE^{\QQ_{T_2}}[\hat\Phi(X_1, Y_1)],\qquad
\sup_{\QQ_{T_2}\in\mathcal{Q}(\mu,\nu)}
\EE^{\QQ_{T_2}}[\hat\Phi(X_1, Y_1)].
\end{eqnarray}

On the other hand, the products available for trading in the market suggest the following semi-static hedging strategies, again expressed in discounted terms:
\begin{equation}\label{eq.repl}  
\phi(X_1) +  \psi(Y_2) + \Delta(X_1,Y_1) (Y_2 - Y_1).
\end{equation}
In view of the definition of the discounted assets in \eqref{eq:T2discount}, the static parts $\phi(X_1)$, $\psi(Y_2)$ correspond to vanilla options written on $p(T_1, T_2)$ and $p(T_2, T_3)$, respectively. As usual these could be approximated using call/put options. 
 The dynamic part $ \Delta(X_1,Y_1) (Y_2 - Y_1)$ corresponds to rebalancing the portfolio in $T_1$ in a self-financing way, between $T_2$- and $T_3$-bonds, according to $\Delta$.
 The quantity in \eqref{eq.repl} then amounts to the value of such a portfolio at time $T_2$. In fact, we find that it suffices to rebalance at time $T_1$ only depending on the value of $Y_1$, that is, it suffices to consider strategies of the form $\Delta(Y_1)$ rather than $\Delta(X_1,Y_1)$,
see Theorem~\ref{SuperRepIntro}.

\subsection{Robust pricing of caplets: primal and dual optimizers}\label{sect.caplets}
In this section we consider vanilla derivatives of the type $\Phi=(p(T_1,T_3)-K)^+$, and comment on their primal and dual optimizers. W.l.o.g.\ we can consider $K=1$. In discounted terms, this leads to considering the following function in \eqref{eq.phidisc}:
\[
\hat\Phi(x,y)=(y-x)^+.
\]
By Proposition~\ref{WTFTheorem}, the robust pricing bounds in \eqref{eq.rob_bounds} can be computed via
\begin{align}\label{eq.bounds.caplets.l}
& \inf_{\pi \in \Pi(\mu, \nu)}\int (b(\pi_x)-x)^+\, \mu(dx),\\
& \sup_{\pi \in \Pi(\mu, \nu)}\int (b(\pi_x)-x)^+\, \mu(dx),
\label{eq.bounds.caplets.u}
\end{align}
respectively.
In Section~\ref{app.dualopt} we discuss in detail the WOT for cost functions of the form $C(x,p)=\theta(b(p)-x)$, where $\theta $ is convex or concave. In particular we will determine the primal as well as dual optimizers and thus  obtain the solutions to the problems \eqref{eq.bounds.caplets.l} and \eqref{eq.bounds.caplets.u}:
\begin{itemize}
\item \emph{Lower bound.}
 The optimizer for the lower bound \eqref{eq.bounds.caplets.l} is determined by the so called  \emph{weak monotone rearrangement} of $\mu$ and $\nu$ (see \cite{AlCoJo17b} and \cite{BaBePa19} for details): Briefly, by using $\ast$ to indicate a push-forward measure, and $\leq_c$ to denote convex order between measures\footnote{Two measures $\nu',\nu \in \Pc_1(\R^d)$ are in convex order, denoted by $\nu' \leq_c \nu$, if $\tint\theta d\nu' \leq \tint\theta d\nu$ for all convex $\theta \colon \R^d \to \R$.}, the family
 $$\{T: \R\to\R, T \text{ is monotone, 1-Lip}, T_\ast(\mu)\leq_c \nu\}$$ has a unique element $\hat T$ for which the $1$-Wasserstein distance $\mathcal W_1(\mu, \hat T_*(\mu))$ is minimized.  
 \eqref{eq.bounds.caplets.l} is then minimized if we put $Y_1=\hat T(X_1)$ and couple $Y_1$ and $Y_2$ by an arbitrary martingale coupling of $T_*(\mu)$ and $\nu$.
 In particular, in this market model, $X_1=\frac{1}{p(T_1,T_2)}$ and $Y_1=\tfrac{p(T_1,T_3)}{p(T_1,T_2)}$ are comonotone.
 
 The corresponding sub-replication strategy is determined in Example~\ref{exa.call}, \eqref{eq.trip.1}.
 It consists in the static position of being short a call on $X_1$ (that is, short puts on $p(T_1,T_2)$) and holding a call on $Y_2$ (that is, holding a call on $p(T_2,T_3)$), and on rearranging the portfolio at time $T_1$ (between the $T_2$ and the $T_3$ bonds) 
in order to hold one unit of the $T_3$-bond if $Y_1$ is above a certain threshold and else investing all in the $T_2$-bond.

\item \emph{Upper bound.} 
On the other hand, the upper bound \eqref{eq.bounds.caplets.l} is obtained in a market where  $Y_1=Y_2$ a.s.\ and where $X_1, Y_1$ are coupled through the \emph{anticomonotone coupling}. In particular,
\[
p(T_1,T_3)= p(T_1,T_2)p(T_2,T_3)\quad \text{a.s.}
\]

The corresponding superreplication strategy  (Example~\ref{exa.call},  \eqref{eq.trip.2})  consists in a static position of holding a put on $X_1$ (that is, long calls on $p(T_1,T_2)$) and being short a put on $Y_2$ (that is, short a put on $p(T_2,T_3)$), and on rearranging the portfolio at time $T_1$ (between the $T_2$- and the $T_3$-bonds) 
in order to be one unit short in the $T_3$-bond if $Y_1$ is below a certain threshold and else investing all in the $T_2$-bond.
\end{itemize}

\section{Weak Optimal Transport problem}\label{sect.WOT}

For the convenience of the reader who is potentially unfamiliar with the mathematical finance aspect of this work, this section can be read independently of the previous sections.
The optimal transport problem for weak transport costs (WOT) was initiated by  \cite{GoRoSaTe17} motivated by applications in geometric inequalities.
Their contribution kicked off a vivid research activity by various groups: \cite{GoRoSaSh18}, \cite{AlCoJo17a}, \cite{AlCoJo17b}, \cite{GoJu18}, \cite{AlBoCh18}, \cite{Sh16}, \cite{BaBePa18}, \cite{BaBePa19}, \cite{BaPa19}, \cite{BaPa20}, \cite{FaGoPr20}.

In this section we use techniques established in WOT to prove our main results ({stated in the introduction}). 
On the way, we prove some auxiliary results that might be of independent interest.

Specifically, we will show in Theorem \ref{thm:convexify} that, for $\mu, C$ continuous, the WOT transport problem can be reduced  to the case where the cost function is convex in the second argument. We prove that 
\begin{equation}
		\label{eq:primal cvx hull no atoms 0}
		\inf_{\pi \in \Pi(\mu,\nu)} \int_{\X} C(x,\pi_x) \, \mu(dx) = \inf_{\pi \in \Pi(\mu,\nu)} \int_\X C^{\ast\ast}(x, \pi_x) \, \mu(dx),
	\end{equation}
where for each $x$, the function $ C^{\ast\ast}(x, \cdot)$ denotes the convex hull of the function $p\mapsto C(x,p)$. This reduction is relevant since virtually all results obtained in WOT so far assume convexity in the second argument.

In Subsection \ref{ss:barycentric} we consider a particular class of WOT problems, where the cost function is barycentric.
We draw a connection of barycentric costs to a specific class of OT problems where the second marginal is not fixed, but has to satisfy a convex order constraint.
Moreover we derive duality results for semi-continuous costs of barycentric type that closely resemble the classical Monge-Kantorovich duality.

\begin{remark}
Following  mathematical finance tradition, we focused on ``super-hedging'' rather than on ``sub-hedging''-results in the introductory Sections 1 and 2.  In contrast, optimal transport problems are usually phrased  as minimization problems.   In the case of the weak transport problem, this sign-convention also matters for the formulation of the convexity assumption on the cost function and for the dual problem that is  based on a specific $\inf$-convolution. We will thus switch  sign in this section and hope that this does not cause  confusion. 
\end{remark}

\subsection{Notation and setting} \label{ss:notation}

Fix $r \in [1,\infty)$.
We denote by $\X$ and $\Y$ Polish spaces with compatible metrics $d_\X$ and $d_\Y$.
We denote by $\Pc_r(\X)$ the set of probability measures on $\X$ which finitely integrate $x \mapsto d_\X(x,x_0)^r$ (for some $x_0 \in \X$ and therefore also for any).
The $r$-Wasserstein distance $\WA_r(\mu,\nu)$ is defined for probabilities $\mu,\nu \in \Pc_r(\X)$
by
\[	\WA_r(\mu,\nu)^r = \inf_{ \pi \in \Pi(\mu,\nu) } \int_{\X \times \X} d_\X(x,x')^r \, \pi(dx,dx'),	\]
and we equip $\Pc_r(\X)$ with the topology induced by $\WA_r$.
The set of continuous functions which are absolutely bounded by a multiple of $1 + d_\X(x,x_0)^r$ is denoted by $\mathcal C_r(\X)$.
For products $\X \times \Y$ of Polish spaces, we will consider the metric $d_r((x,x'), (y,y'))^r = d_\X(x,x')^r + d_\Y(y,y')^r$ and the corresponding Wasserstein space $\Pc_r(\X \times \Y)$.
We will often use the notation $\mu(f)$ to abbreviate $\int_\X f(x) \, \mu(dx)$.

\subsubsection{The weak optimal transport problem}
Consider a measurable cost function $C \colon \X \times \Pc(\Y) \to (-\infty,\infty]$ satisfying the standing assumption that there exist  a constant $K > 0$ and $(x_0,y_0) \in \X \times \Y$ with
\begin{equation} \label{eq:lower bound}
C(x,p) \geq - K \left( 1 + d_\X( x, x_0 )^r + \WA_r(p, \delta_{y_0} )^r \right).
\end{equation}
Then the weak transport problem with cost $C$ between $\mu \in \Pc_r(\X)$ and $\nu \in \Pc_r(\Y)$
is given by
\begin{equation} \label{WOT} \tag{WOT}
	\inf_{\pi \in \Pi(\mu,\nu)} \int_\X C(x,\pi_x) \, \mu(dx).
\end{equation}
Note that \eqref{eq:lower bound} guarantees that \eqref{WOT} is well-defined.
Since the disintegration of a measure is measurable but not (weakly) continuous, we observe an inherent asymmetry in \eqref{WOT}.
To be able to obtain basic existence and duality results, different additional conditions  have been  imposed on  the cost function $C$ in the literature. In particular it is generically assumed that $C$ is  lower semicontinuity and satisfies
\begin{equation}
	\label{eq:convexity assumption}
	\forall x \in \X \colon p \mapsto C(x,p)\text{ is convex}.
\end{equation}
 In \cite{BaBePa18} it is  established that these conditions are in fact sufficient to guarantee existence of a minimizer as well as a duality. 
The key observation in \cite{BaBePa18}  is that the spaces $\X$ and $\Y$ play different roles: $\X$ is akin to the state space of a stochastic process at time 1 and $\Y$ akin to the state space at time 2.  This viewpoint leads to a technical convenient relaxation of \eqref{WOT} that we discuss in the next section.

\subsubsection{The relaxation of WOT}
We start by defining an embedding of $\Pc(\X \times \Y) $ into  $\Pc(\X \times \Pc(\Y))$. To this end we write 
$\proj^i$ for the projection onto the $i$-th component and consider the disintegration $(\pi_x)_{x\in \X }$ of  a measure $\pi \in \Pc(\X \times \Y) $ w.r.t.\ its first marginal $\proj^1_*(\pi)$. Specifically we define 
\[ J(\pi) := (x \mapsto (x,\pi_x))_\ast\left( \proj^1_\ast(\pi) \right) \in \Pc(\X \times \Pc(\Y)). \]
The embedding $J$ forms the basis for the relaxed version of \eqref{WOT}.
As we are interested in probabilities / processes with prescribed marginals $\mu,\nu$, we need to  make sense of what it means for a measure in $\Pc(\X \times \Pc(\Y))$ to have $\nu$ as its ``second marginal''.
To this end, consider the intensity maps $\hat I \colon \Pc(\X \times \Pc(\Y)) \to \Pc(\X \times \Y)$ and $I \colon \Pc(\Pc(\Y)) \to \Pc(\Y)$, where for $P \in \Pc(\X \times \Pc(\Y))$ and $Q \in \Pc(\Pc(\Y))$ their intensities $\hat I(P)$ and $I(Q)$ are given by the unique measures satisfying
\begin{gather}
	\label{eq:def_hat I}
	\hat I(P)(f) = \int_{\X \times \Pc(\Y)} \int_{\Y} f(x,y)\, p(dy) \, P(dx,dp) \quad \forall f \in \mathcal C_b(\X \times \Y), \\
	\label{eq:def_I}
	I(Q)(f) = \int_{\Pc(\Y)} \int_\Y f(y) \, p(dy) \, Q(dp)\quad \forall f \in \mathcal C_b(\Y),
\end{gather}
respectively.
Note that $\hat I$ is the left inverse of $J$, i.e.\ $\hat I(J(\pi)) = \pi$ for any $\pi \in \Pc(\X \times \Y)$.
We denote by $\Pc(\X \leadsto \mathcal Z)$ the set of probability measures on $\X \times \mathcal Z$ concentrated on the graph of a measurable function.
Then $J$ maps onto $\Pc(\X \leadsto \Pc(\Y))$, and $\hat I$ restricted to $\Pc(\X \leadsto \Pc(\Y))$ is its inverse.

We define $\Lambda(\mu,\nu) \subset \Pc(\X \times \Pc(\Y))$ by
\begin{equation}
	\label{eq:def_Lambda}
	\Lambda(\mu,\nu) := \left\{ P \in \Pc(\X \times \Pc(\Y)) \colon \hat I(P) \in \Pi(\mu,\nu) \right\}.
\end{equation}
Since $\proj^1_\ast(\hat I(P)) = \proj^1_\ast P$ and $\proj^2_\ast ( \hat I(P)) = I( \proj^2_\ast P)$, we obtain
\begin{equation}
	\label{eq:def_Lambda2}
	\Lambda(\mu,\nu) = \left\{ P \in \Pc(\X \times \Pc(\Y)) \colon \proj^1_\ast P = \mu, I(\proj^2_\ast P)
	= \nu \right\} = \bigcup_{\substack{\nu' \in \Pc(\Pc(\Y)) \\ I(\nu') = \nu} } \Pi(\mu,\nu').
\end{equation}
Following \cite{BaBePa18}, we consider the following relaxation of \eqref{WOT}
\begin{equation}
	\label{WOT'} \tag{WOT'}
	\inf_{P \in \Lambda(\mu,\nu)} \int_{\X \times \Pc(\Y)} C(x,p) \, P(dx,dp).
\end{equation}
Clearly, we have that $J(\Pi(\mu,\nu)) \subset \Lambda(\mu,\nu)$ and hence \eqref{WOT'}$\le$\eqref{WOT}, but in many cases of interest more can be said.
As long as $\mu$ is atomless, we have by the next lemma that $J(\Pi(\mu,\nu))$ is dense in $\Lambda(\mu,\nu)$ and hence the values of the two problems coincide under mild regularity assumptions. 

\begin{proposition}\label{prop:couplings.dense}
	Let $\mu \in \Pc(\X)$ be continuous and $\nu \in \Pc(\Y)$.
	Then the set $J(\Pi(\mu,\nu)) = \Lambda(\mu,\nu) \cap \Pc(\X \leadsto \Pc(\Y))$ is dense in $\Lambda(\mu,\nu)$.
If in addition $\mu \in \Pc_r(\X), \nu \in \Pc_r(\Y)$, and $C \in \mathcal C_r(\X \times \Pc_r(\Y))$, the values of \eqref{WOT} and \eqref{WOT'} coincide.
\end{proposition}

\begin{proof}
It is well-known, see e.g.\ \cite{Pr07b}, that, for a  continuous probability measure $\mu$ on some Polish space $\X$ and an arbitrary probability measure $\nu'$ on some Polish space $\mathcal Z$, the set of Monge couplings, that is $\Pi(\mu,\nu') \cap \Pc(\X \leadsto \mathcal Z)$, is dense in $\Pi(\mu,\nu')$.
	Applying this to $\mathcal Z = \mathcal P(\Y)$ and $\nu' \in \Pc(\Pc(\Y))$ yields that $\Pi(\mu,\nu') \cap \Pc(\X \leadsto \Pc(\Y))$ is dense in $\Pi(\mu,\nu')$. Hence, \eqref{eq:def_Lambda2} yields that
	\begin{align*}
		\Lambda(\mu,\nu) \cap \Pc(\X \leadsto \Pc(\Y)) =
		\bigcup_{\substack{ \nu' \in \Pc(\Pc(\Y))\\ I(\nu') = \nu}} \Pi(\mu,\nu') \cap \Pc(\X \leadsto \Pc(\Y))
	\end{align*}
	is dense in $\Lambda(\mu,\nu)$.
	Finally, since $\hat I$ restricted to $\Pc(\X \leadsto \Pc(\Y))$ is the inverse of $J$, we have that $J \colon \Pi(\mu,\nu) \to \Pc(\X \leadsto \Pc(\Y)) \cap \Lambda(\mu,\nu)$ is a bijection and $J(\Pi(\mu,\nu)) = \Pc(\X \leadsto \Pc(\Y)) \cap \Lambda(\mu,\nu)$.
Now let $\mu\in \Pc_r(\X)$, $\nu \in \Pc_r(\Y)$, and $C \in \mathcal C_r(\X \times \Pc_r(\Y))$.
	For $P \in \Lambda(\mu,\nu)$ we find a sequence $P^k \in \Lambda(\mu,\nu) \cap \Pc_r(\X \leadsto \Pc_r(\Y))$ that converges to $P$ weakly. 
	Since the marginals are fixed with finite $r$-th moments, Definition 6.8 in \cite{Vi09} yields that this convergence holds even in $\WA_r$
	on $\Pc_r(\X \times \Pc_r(\Y))$.
	Let $\pi^k \in \Pi(\mu,\nu)$ with $J(\pi^k) = P^k$, $k \in \N$. Then
	\begin{align*}
	\inf_{\pi \in \Pi(\mu,\nu)} \int_\X C(x,\pi_x) \, \mu(dx) &\leq \liminf_{k \to \infty} \int_\X C(x,\pi^k_x) \, \mu(dx) \\
	&= \lim_{k \to \infty} \int_{\X \times \Pc_r(\Y)} C(x,p) \, P^k(dx,dp) = \int_{\X \times \Pc_r(\Y)} C(x,p) \, P(dx,dp).
	\end{align*}
	As $P$ was an arbitrary element of $\Lambda(\mu,\nu)$ and \eqref{WOT}$\ge$\eqref{WOT'}, this shows that the values of \eqref{WOT} and \eqref{WOT'} coincide.
\end{proof}

\subsubsection{On the process interpretation of  the weak transport problem and its relaxation}\label{sect.inter}
A common interpretation of probability measures $\pi \in \Pi(\mu,\nu)$ in optimal transport is to view $\pi$ as the law of a two step stochastic process $(X_1, X_2)=(X, Y)$ where $X \sim \mu$ and $Y \sim \nu$.
Due to the inherent asymmetry of \eqref{WOT}, this point of view seems even more natural in the framework of weak optimal transport.
Specifically, the usual weak transport problem can be reformulated as an optimization problem over (2-period) stochastic processes  as indicated by the following equality:
\begin{align}\label{eq:WOT2proc1} \inf_{\pi \in \Pi(\mu,\nu)} \int_{\X} C(x,\pi_x) \, \mu(dx) = 	\inf_{ \substack{(\Omega, \P, X,Y) \\ X \sim \mu, Y \sim \nu}} \mathbb E_\P[C(X, \Law(Y|X))],\end{align}
where $\Law$ denotes the law under $\P$.
Notably, also the relaxed problem introduced in the previous subsection admits a natural formulation in terms of \emph{filtered processes}, that is, processes together with a filtration. Indeed, relaxing  the set $\Pi(\mu,\nu)$ on the left-hand side of \eqref{eq:WOT2proc1} corresponds on the right-hand side of \eqref{eq:WOT2proc1} to minimizing over adapted processes together with arbitrary filtrations, rather than allowing only for filtrations generated by the process itself: 
\begin{proposition}\label{pr:ProcessInterpretation} Let $C: \X\times \Pc(\Y)\to \R$ be measurable. Then we have
\begin{align}\label{eq:WOT2proc2} \inf_{P \in \Lambda(\mu,\nu)} \int_{\X} C(x,p) \, P(dx,dp) = 	\inf_{\substack{(\Omega, \P,(\F_t)_{t=1,2},X,Y) \\ X \sim \mu, Y \sim \nu}} \E_\P[C(X,\Law(Y|\F_1))],\end{align}
in the sense that either side is defined if the other is and then the two are equal. 
\end{proposition}

To obtain \eqref{eq:WOT2proc2} we need to establish the appropriate correspondence between elements of $     \Lambda(\mu,\nu)   $ and stochastic processes with marginals $\mu,\nu$, which is what we do in the proof of the proposition.
\begin{proof}[Proof of Proposition \ref{pr:ProcessInterpretation}]
To show `$\leq$' assume that $(\Omega, \F, \P)$ is some probability space and that $(X, Y)$, $X \sim \mu, Y\sim \nu$, is  a process adapted to a filtration $(\F_t)_{t=1,2}$.  Then this process induces the required candidate in $\Lambda(\mu, \nu)$ by setting 
\begin{align}P:= \Law(X, \Law(Y|\F_1)).
\end{align}
To show the other inequality `$\geq $', fix $P \in \Lambda(\mu,\nu)$.
	Define a filtered probability space by $\X \times \Pc(\Y) \times \Y$ with $\sigma$-algebra $\F = \mathcal B(\X) \otimes \mathcal B(\Pc(\Y)) \otimes \mathcal B(\Y)$, probability measure $\mathbb P(dx,dp,dy) = P(dx,dp)p(dy)$, and filtration $\F_1 = \sigma((x,p,y) \mapsto (x,p))$, $\F_2 = \F$.
	Then the process  $(X(x,p,y) = x, Y(x,p,y) = y)$ is as required on the r.h.s. of \eqref{eq:WOT2proc2}.
\end{proof}
\newcommand{\AW}{\mathcal {AW}}
\newcommand{\FFP}{\mathrm{FP}}
\begin{remark}
We briefly describe the wider context of the relation between $\Pc(\X \times \Pc(Y))$ and the filtered processes which is the topic of \cite{BaBePa21}. There it is detailed that the space $\FFP$ of filtered (stochastic)  processes is naturally equipped with an adapted variant of the Wasserstein distance denoted by $\AW$. Two processes have distance $0$ if and only if they have the same probabilistic properties in a specific sense. Identifying such processes, the space $(\FFP, \AW) $ is a complete metric space, isometric to a classical Wasserstein space $(Z, \WA_Z)$. Specifically, in the case of two periods,
\[ (\Pc(\X\times \Pc(\Y)) , \WA) \quad \text{is isometric to} \quad (\FFP, \AW) \]
and the corresponding isometry (and its inverse, respectively) are the operations given in the proof of Proposition \ref{pr:ProcessInterpretation}. Stochastic processes where the filtration is generated by the process itself  correspond to elements of $J(\Lambda(\mu, \nu))$. Proposition \ref{prop:couplings.dense} then asserts precisely that such processes with given marginals are dense in the set of general processes with these marginals.
\end{remark}

\subsection{Weak transport duality for non-convex cost} \label{ss:WOT duality}

The main result of this section is Theorem~\ref{thm:convexify}, which can be seen as a justification for the convexity assumption on the cost $C$ (in the sense of \eqref{eq:convexity assumption}) that is typically used in the WOT-literature.
First we show, in Proposition~\ref{prop:convexify}, that the relaxed problem \eqref{WOT'} for a cost $C$ coincides with a weak transport problem \eqref{WOT} for its convexification $C^{\ast\ast}$.
Then in Theorem~\ref{thm:convexify} we obtain that the values of \eqref{WOT} for cost $C$ and the convexification $C^{\ast\ast}$ coincide, provided that the cost is continuous and the first marginal is continuous as well.
Hence, when one is solely interested in the value of \eqref{WOT} for continuous cost, it is of no harm to convexify $C$ and thus work in the framework of classical WOT theory.

For the sake of completeness we include the following (slightly more general) version of the first assertion of \cite[Theorem 3.1]{BaBePa18}, where we weaken the assumption that $C$ is bounded from below to \eqref{eq:lower bound}.
So far the weak transport duality was usually stated in the form of \eqref{eq:duality 1} using an $\inf$-convolution instead of pairs of dual functions, c.f.\ \cite[Theorem 9.6]{GoRoSaTe17}, \cite[Theorem 4.2]{AlBoCh18} and \cite[Theorem 3.1]{BaBePa18}.
A minor contribution of the next theorem is the harmonisation of the weak transport duality and the Kantorovich duality (known from classical optimal transport) in terms of dual pairs instead of utilizing an $\inf$-convolution.

\begin{theorem}[Duality] \label{thm:duality}
	Let $(\mu,\nu) \in \Pc_r(\X) \times \Pc_r(\Y)$ and $C \colon \X \times \Pc_r(\Y) \to (-\infty,\infty]$ be lower semicontinuous such that 
	Then
	\begin{align} \label{eq:duality 1}
	\inf_{P \in \Lambda(\mu,\nu)} \int C(x,p) \, P(dx,dp) &= \sup_{\psi \in \mathcal C_r(\Y)} \nu(\psi) + \mu(R_C\psi) \\ \label{eq:duality 2}
	&= \sup_{ \substack{ \phi \in \mathcal C_r(\X),\, \psi \in \mathcal C_r(\Y) \\ \phi(x) + p(\psi) \leq C(x,p) }} \mu(\phi) + \nu(\psi),
	\end{align}
	where $R_C\psi(x) := \inf \left\{-p(\psi) + C(x,p) \colon p \in \Pc_r(\Y)\right\}$ and $\mu(R_C\psi):= -\infty$ if the integral of $R_C \psi$ w.r.t.\ $\mu$ is not well-defined.
\end{theorem}

\begin{proof}
	Define the auxiliary cost function
	 $C^K \colon \X \times \Pc(\Y) \to [0,\infty)$ by
	\[ C^K(x,p) := C(x,p) + K\left( 1 + d_\X(x,x_0)^r + \WA_r(p,\delta_{y_0})^r \right),\]
	where $(x,p) \in \X \times \Pc(\Y)$.
	For $P \in \Lambda(\mu,\nu)$ we have
	\begin{align*}
	& \int_{\X \times \Pc(\Y)} 1 + d_\X(x,x_0)^r + \WA_r(p,\delta_{y_0})^r \, P(dx,dp) \\
	&\phantom{  \int 1 + d_\X(x,x_0) + } = 1 + \int_\X d_\X(x,x_0)^r \, \mu(dx) + \int_{\X \times \Pc(\Y)} \int_\Y d_\Y(y,y_0)^r \, p(dy) \, P(dx,dp) \\
	&\phantom{  \int 1 + d_\X(x,x_0) + } = 1 + \int_\X d_\X(x,x_0)^r \, \mu(dx) + \int_\Y d_\Y(y,y_0)^r \, \nu(dy),
	\end{align*}
	and therefore
	\begin{multline*}
	\int_{\X \times \Pc(\Y)} C^K(x,p) \, P(dx,dp) = \int_{\X \times \Pc(\Y)} C(x,p) \, P(dx,dp) \\ + K \left( 1 + \int_\X d_\X(x,x_0)^r \, \mu(dx) + \int_\Y d_\Y(y,y_0)^r \, \nu(dy) \right).
	\end{multline*}
	For $\psi \in \mathcal{C}_r(\Y)$ we write
	\[ \psi_K(y) := \psi(y) - K \left( 1 + d_\Y(y,y_0)^r \right) \in \mathcal C_r(\Y) \]
	and obtain the following relation between the inf-convolutions of $C$ and $C^K$:
	\begin{align*}
	R_{C^K} \psi(x) &= \inf_{p \in \Pc_r(\Y)} -p(\psi) + C^K(x,p) \\
	&= \inf_{p \in \Pc_r(\Y)} -p(\psi) + C(x,p) + K\left( 1 + d_\X(x,x_0)^r + \int_\Y d_\Y(y,y_0)^r \, p(dy)\right) \\
	&= \inf_{p \in \Pc_r(\Y)} -p(\psi_K) + C(x,p) + K d_\X(x,x_0)^r \\
	&= R_C \psi_K(x) + Kd_\X(x,x_0)^r.
	\end{align*}
	In particular, the $\mu$-integral of $R_{C^K}\psi$ is well-defined if and only if the $\mu$-integral
	of $R_C\psi_K$ is.
	Since $C^K$ is bounded from below and lower semicontinuous, \cite[Theorem 3.1]{BaBePa18} yields
	\footnote{The supremum in the duality of \cite[Theorem 3.1]{BaBePa18} is taken over all functions $\psi
	\in \mathcal C_r(\Y)$ which are bounded from above. Therefore, $R_{C^K}\psi$ is bounded from below 
	and its $\mu$-integral is well-defined. We circumvent this by defining the $\mu$-integral of
	$R_{C^K}\psi$ appropriately, that is $-\infty$, whenever it is not well-defined.}
	\begin{align*}
	\inf_{P \in \Lambda(\mu,\nu)} \int_{\X \times \Pc(\Y)} C^K(x,p) \, P(dx,dp) &= \sup_{\psi \in \mathcal C_r(\Y)} \nu(\psi) + \mu(R_{C^K}\psi) \\
	&= \sup_{\psi \in \mathcal C_r(\Y)} \nu(\psi) + \int_\X R_C \psi_K(x) + Kd_\X(x,x_0)^r \, \mu(dx) \\
	&= K\left(1 + \int_\X d_\X(x,x_0)^r \, \mu(dx) + \int_\Y d_\Y(y,y_0)^r \, \nu(dy)\right) \\
	&\phantom{=}+ \sup_{\psi \in \mathcal C_r(\Y)} \nu(\psi_K) + \mu(R_C\psi_K).
	\end{align*}
	Rearranging the terms and relabeling $\psi_K$ as $\psi$ yields \eqref{eq:duality 1}.
	
	The next goal is to show \eqref{eq:duality 2}, which resembles more closely the classical Kantorovich duality.
	By the first part of the proof, we may assume w.l.o.g.\ that $C$ is nonnegative.
	It is well-known that for any lower semicontinuous, nonnegative function $f$ (on a metric space), there exists a sequence of nonnegative functions $(f_k)_{k \in \N}$ such that $f_k$ is $k$-Lipschitz continuous and absolutely bounded by $k$, and $f_k \nearrow f$.
	
	Let $(C^K)_{k \in \N}$ be such a sequence for $C$.
	Monotone convergence and existence of minimizers, see \cite[Theorem 2.9]{BaBePa18}, yield
	\begin{align*}
	\sup_k \inf_{P \in \Lambda(\mu,\nu)} \int_{\X \times \Pc(\Y)} C^K(x,p) \, P(dx,dp) = \inf_{P \in \Lambda(\mu,\nu)} \int_{\X \times \Pc(\Y)} C(x,p) \, P(dx,dp).
	\end{align*}
	Assume for a moment that \eqref{eq:duality 2} holds for all $k \in \N$, then
	\begin{align*}
	\inf_{P \in \Lambda(\mu,\nu)} \int_{\X \times \Pc(\Y)} C(x,p) \, P(dx,dp) &= \sup_k \sup_{ \substack{ \phi \in \mathcal C_r(\X),\, \psi \in \mathcal C_r(\Y) \\ \phi(x) + p(\psi) \leq C^K(x,p) } } \nu(\psi) + \mu(\phi) \\
	&\leq \sup_{ \substack{ \phi \in \mathcal C_r(\X),\, \psi \in \mathcal C_r(\Y) \\ \phi(x) + p(\psi) \leq C(x,p) } } \nu(\psi) + \mu(\phi) \\
	&\leq \inf_{P \in \Lambda(\mu,\nu)} \int C(x,p) \, P(dx,dp).
	\end{align*}
	Therefore, the inequalities are in fact equalities and \eqref{eq:duality 2} holdsh for $C$ as in the statement of the theorem.
	
	It is sufficient to show \eqref{eq:duality 2} for $k$-Lipschitz weak transport costs $C$, $k\in\N$.
	When $C$ is $k$-Lipschitz, its $\inf$-convolution $R_C\psi$ is also $k$-Lipschitz as the next 
	computation shows:
	\[ |R_C\psi(x) - R_C\psi(x')| \leq \sup_{p \in \Pc_r(\Y)} |C(x,p) - C(x',p)| \leq k d_\X(x,x').\]
	As a consequence we have that $R_C\psi \in \mathcal C_r(\X)$, and
	\begin{align*}
	\sup_{\psi \in \mathcal C_r(\Y)} \nu(\psi) + \mu(R_C\psi) &\geq \sup_{ \substack{ \phi \in \mathcal C_r(\X),\, \psi \in \mathcal C_r(\Y) \\ \phi(x)  + p(\psi) \leq C(x,p) } } \nu(\psi) + \mu(\phi).
	\end{align*}
	On the other hand, the reverse inequality is easy to show:
	Let $\phi \in \mathcal C_r(\X)$ and $\psi \in \mathcal C_r(\Y)$ such that $\phi(x) + p(\psi) \leq C(x,p)$.
	Then $\phi(x) \leq -p(\psi) + C(x,p)$ for all $(x,p) \in \X \times \Pc_r(\Y)$, hence,
	\[ \phi(x) \leq \inf_{p \in \Pc_r(\Y)} -p(\psi) + C(x,p) = R_C \psi(x)\quad \forall x\in \X. \]
	Therefore we have
	\[ \nu(\psi) + \mu(\phi) \leq \nu(\psi) + \mu(R_C\psi), \]
	which readily shows the reverse inequality and completes the proof.
\end{proof}

The following result represents a version of Jensen's inequality that proves to be useful in our context.
\begin{lemma}
	\label{lem:Jensen}
	Let $F\colon \Pc_r(\Y) \to (-\infty,\infty]$ be convex.
	Assume that either of the following holds:
	\begin{enumerate}[label = (\alph*)]
		\item \label{it:jensen lsc}
			$F$ is lower semicontinuous;
		\item \label{it:jensen usc}
			$F$ is upper semicontinuous and upper bounded by a multiple of  $p \mapsto \WA_r(p,p_0)^r$ for some $p_0 \in \Pc_r(\Y)$.
	\end{enumerate}
	Then we have for $Q \in \Pc_r(\Pc_r(\Y))$
	\[	F(I(Q)) \leq \int_{\Pc_r(\Y)} F(q) \, Q(dq). \]
\end{lemma}

\begin{proof}
	Fix $Q \in \Pc_r(\Pc_r(\Y))$. 
	If \ref{it:jensen lsc} is satisfied, for $p \in \Pc_r(\Y)$ by Fenchel's duality theorem we have
	\[	F(p) = \sup_{\psi \in \mathcal C_r(\Y)} p(\psi) - \sup_{q \in \Pc_r(\Y)} q(\psi) - F(q). \]
	Therefore, for $Q \in \Pc_r(\Pc_r(\Y))$ with $I(Q) = p$, by interchanging supremum with integration, we have that
	\[	F(p) = \sup_{\psi \in \mathcal C_r(\Y)} \int_{\Pc_r(\Y)} q'(\psi) \, Q(dq')  - \sup_{q \in \Pc_r(\Y)} q(\psi) - F(q) \le \int_{\Pc_r(\Y)} F(q') \, Q(dq').	\]

	Now assume that \ref{it:jensen usc} holds.	
	Since $\Pc_r(\Y)$ is Polish, there exists by \cite[Lemma 3.22]{Ka02} a map 
	$h \colon (0,1) \to \Pc_r(\Y)$ such that $Q = h_\ast \lambda$,
	where $\lambda$ denotes the Lebesgue measure on $(0,1)$.
	We define, for $1 \leq k \leq n \in \N$,
	\[	Q_k^n(dq) := n \int_{\frac{k-1}{n}}^{\frac{k}{n}} \delta_{h(t)} \, dt.	\]
	It is immediate that $\frac1n \sum_{k = 1}^n Q_k^n = Q$, and
	\begin{align*}
		\WA_r\left( \frac1n \sum_{k = 1}^n \delta_{I(Q_k^n)},Q\right)^r
		&= \WA_r\left( \frac1n \sum_{k = 1}^n \delta_{n \int_{\frac{k-1}{n}}^{\frac{k}{n}} h(t) \, dt},
		\int_0^1 \delta_{h(t)} \, dt \right)^r
	\\	&\leq \frac1n \sum_{k=1}^n \WA_r\left( \delta_{n \int_{\frac{k-1}{n}}^{\frac{k}{n}} h(t) \, dt}, n \int_{\frac{k-1}{n}}^{\frac{k}{n}} \delta_{h(t)} \, dt \right)^r
	\\	&\leq \sum_{k=1}^n \int_{\frac{k-1}{n}}^{\frac{k}{n}}\WA_r\left( n \int_{\frac{k-1}{n}}^{\frac{k}{n}} h(t)\, dt, h(s) \right)^r \, ds
	\\	&\leq n \sum_{k=1}^n \int_{\frac{k-1}{n}}^{\frac{k}{n}} \int_{\frac{k-1}{n}}^{\frac{k}{n}} \WA_r(h(t),h(s))^r \, dt \, ds
	\\	&\leq \sup\left\{ \int_{(0,1)^2} \WA_r(h(t),h(s))^r \chi(dt,ds) \colon \chi \in \Pi(\lambda,\lambda) \right. 
	\\ & \phantom{\leq \sup\Big\{ \int_{(0,1)^2} \WA_r(h(t),h(s))^r \chi(dt,ds) \colon } \left. \int_{(0,1)^2} |s - t|^r \, \chi(dt,ds) \leq \frac{1}{n^r} \right\},
	\end{align*}
	where we applied multiple times Jensen's inequality (which is possible in this settings thanks to the first
	part of this lemma). The right-hand side vanishes for $n \to \infty$ by \cite[Lemma 2.7]{Ed19}.
	Due to convexity of $F$, we have
	\[
		 F(I(Q)) = F\left(\frac1n \sum_{k = 1}^n I(Q_k^n)\right) \leq \frac1n \sum_{k=1}^n F(I(Q_k^n)).
	\]
	We conclude by taking the limit superior for $n \to \infty$, which yields by upper semicontinuity
	\[
		F(I(Q)) \leq \limsup_{n \to \infty} \frac1n \sum_{k = 1}^n F(I(Q_k^n)) \leq \int_{\Pc_r(\Y)} F(q) \, Q(dq). \qedhere
	\]
\end{proof}

\begin{lemma}
	\label{lem:CastleqtildeC}
	Assume that $C \colon \X \times \Pc_r(\Y) \to (-\infty,\infty]$ is measurable and bounded from below as in \eqref{eq:lower bound}. 
	For $x \in \X$ we denote by $C^{\ast\ast}(x,\cdot)$ the lower semicontinuous convex envelope of $C(x,\cdot)$.
	For $x \in \X$, $p \in \Pc_r(\Y)$ we write $\tilde C(x,p):=\inf_{Q \in \Pc_r(\Pc_r(\Y)), \, I(Q) = p}  \int C(x,q) \, Q(dq)$.
	Let $Q \in \Pc_r(\Pc_r(\Y))$ with $I(Q) = p$, then
	\begin{equation}
		\label{eq:CastleqtildeC}
		C^{\ast\ast}(x,p) \leq \tilde C(x,p) \leq \int_{\Pc_r(\Y)} \tilde C(x,q) \, Q(dq).
	\end{equation}
	Moreover, if $C(x,\cdot)$ is lower semicontinuous there is equality, i.e.\ $C^{\ast\ast}(x,p) = \tilde C(x,p)$.
\end{lemma}

\begin{proof}
	By Lemma \ref{lem:Jensen} we have for $Q \in \Pc_r(\Pc_r(\Y))$ with $I(Q) = p$ that
	\[	C^{\ast\ast}(x,p)  \le \int_{\Pc_r(\Y)} C^{\ast\ast}(x,q) \, Q(dq) \leq \int_{\Pc_r(\Y)} C(x,q)\, Q(dq).	\]
	Taking the infimum over all $Q \in \Pc_r(\Pc_r(\Y))$ with $I(Q) = p$ shows the first inequality
	in \eqref{eq:CastleqtildeC}.

	To see the second inequality in \eqref{eq:CastleqtildeC}, we observe that the map $q \mapsto
	\tilde C(x,q)$ is by Proposition 7.47 in \cite{BeSh78} lower semianalytic, and that $D = \{(p,Q) \in
	\Pc_r(\Y) \times \Pc_r(\Pc_r(\Y)) \colon I(Q) = p\}$ is closed.
	Therefore, by Proposition 7.50 in \cite{BeSh78}, for each $\epsilon > 0$ there is an analytically measurable
	function $q\mapsto Q^{q}$ with $I(Q^q)=q$ and
	\begin{equation} \label{eq:measurable selection}
		\int_{\Pc_r(\Y)} C(x,q')	\, Q^{q}(dq') \leq \tilde C(x,q) + \epsilon\quad \forall q \in\Pc_r(\Y).
	\end{equation}
	We now compute the intensity of $\hat Q(dq) := \int_{\Pc_r(\Y)} Q^{q'}(dq) \, Q(dq')$:
	\[	I(\hat Q) = \int_{\Pc_r(\Y)} I(Q^{q'}) \, Q(dq') = \int_{\Pc_r(\Y)} q' \, Q(dq') = I(Q) = p.\]
	Hence,
	\[	\tilde C(x,p) \leq \int_{\Pc_r(\Y)} C(x,q) \, \hat Q(dq) \leq \int_{\Pc_r(\Y)} \tilde C(x,q) \, Q(dq) + \epsilon,	\]
	which proves the second inequality in \eqref{eq:CastleqtildeC} as $\epsilon$ is arbitrary.	

	If $C(x,\cdot)$ is lower semicontinuous, we have by Theorem \ref{thm:duality}
	\[ \sup_{\psi \in \mathcal C_r(\Y)} p(\psi) + R_C\psi(x) = \inf_{P \in \Lambda(\delta_x,p)} \int C(x',p) \, P(dx',dp) = \tilde C(x,p), \]
	whence $C^{\ast\ast}(x,p) = \tilde C(x,p)$ for all $p \in \Pc_r(\Y)$.
\end{proof}

\begin{proposition} \label{prop:convexify}
	Let $\mu \in \Pc(\X)$ and $\nu \in \Pc_r(\Y)$.
	Under the assumptions of Lemma \ref{lem:CastleqtildeC}, we have
	\begin{align} \label{eq:convexify wot'}
		\inf_{P \in \Lambda(\mu,\nu)} \int C(x,p) \, P(dx,dp) &= \inf_{P \in \Lambda(\mu,\nu)} \int \tilde C(x,p) \, P(dx,dp) \\ \label{eq:primal tildeC} 
		&= \inf_{\pi \in \Pi(\mu,\nu)} \int \tilde C(x,\pi_x) \, \mu(dx) \\
		&\geq \inf_{P \in \Lambda(\mu,\nu)} \int C^{\ast\ast}(x,p) \, P(dx,dp) \label{eq:primal cvx hull relaxed}\\
		&= \inf_{\pi \in \Pi(\mu,\nu)} \int C^{\ast\ast}(x,\pi_x) \, \mu(dx). \label{eq:primal cvx hull}
	\end{align}
	Moreover, if $C(x,\cdot)$ is lower semicontinuous for all $x \in \X$, there is equality.
\end{proposition}

\begin{proof}
	Let $P \in \Lambda(\mu,\nu)$ and write $\hat I(P) =: \pi \in \Pi(\mu,\nu)$. For disintegrations
	$(P_x)_{x \in \X}$ and $(\pi_x)_{x \in \X}$ of $P$ and $\pi$ w.r.t.\ $\mu$, respectively, we have
	by definition of the intensity maps, c.f.\ \eqref{eq:def_hat I} and \eqref{eq:def_I}, the relation
	\begin{equation} \label{eq:relation of disintegration and intensity}
		I(P_x) = \pi_x \quad \mu\text{-almost surely}.
	\end{equation}
	Due to Lemma \ref{lem:CastleqtildeC} and \eqref{eq:relation of disintegration and intensity}, we have
	\begin{align*}
		\int_{\X \times \Pc_r(\Y)} C(x,p) \, P(dx,dp) &\geq \int_{\X \times \Pc_r(\Y)} \tilde C(x,p) \, P(dx,dp)
		= \int_{\X} \int_{\Pc_r(\Y)} \tilde C(x,p) \, P_x(dp) \, \mu(dx)
	\\	&\geq \int_\X \tilde C(x, \pi_x) \, \mu(dx) \geq \int_\X C^{\ast\ast}(x, \pi_x) \, \mu(dx).
	\end{align*}
	Clearly, \eqref{eq:primal cvx hull relaxed} is dominated by \eqref{eq:primal cvx hull}, and at the
	same time we have by Lemma \ref{lem:Jensen}
	\[	\int_{\X \times \Pc_r(\Y)} C^{\ast\ast}(x,p) \, P(dx,dp) = \int_{\X} \int_{\Pc_r(\Y)} 
		C^{\ast\ast}(x,p) \, P_x(dp) \, \mu(dx) \geq \int_\X C^{\ast\ast}(x,I(P_x)) \, \mu(dx),	\]
	which proves that \eqref{eq:primal cvx hull relaxed} and \eqref{eq:primal cvx hull} coincide.

	It remains to show that the left-hand side of \eqref{eq:convexify wot'} is dominated by
	\eqref{eq:primal tildeC}. Pick $\pi \in \Pi(\mu,\nu)$.
	By the measurable selection argument presented in the proof of Lemma \ref{lem:CastleqtildeC}, we find for $\epsilon > 0$ an analytically measurable map $q \mapsto Q^q$ with $I(Q^q) = q$ and such that \eqref{eq:measurable selection} holds.
	The intensity of $P(dx,dq) := \mu(dx) \, Q^{\pi_x}(dq)$ is given by
	\[	\hat I(P)(dx,dy) = \mu(dx) \, I(Q^{\pi_x})(dy) = \mu(dx) \, \pi_x(dy) = \pi(dx,dy).	\]
	Thus $P(dx,dp) \in \Lambda(\mu,\nu)$ and
	\begin{align*}
		\int_{\X \times \Pc_r(\Y)} C(x,p) \, P(dx,dp) = \int_{\X} \int_{\Pc_r(\Y)} C(x,p) \, Q^{\pi_x}(dp) \, \mu(dx) \leq \int_\X \tilde C(x,\pi_x) \, \mu(dx) + \epsilon.
	\end{align*}
	We see that \eqref{eq:primal tildeC} dominates the left-hand side of \eqref{eq:convexify wot'} by recalling that $\epsilon > 0$ was arbitrary.
	Finally, the last statement follows from Lemma \ref{lem:CastleqtildeC}.
\end{proof}

\begin{theorem}\label{thm:convexify}
	Assume that $\mu \in \Pc_r(\X)$ is continuous, $\nu \in \Pc_r(\Y)$, and $C \in \mathcal C_r(\X \times \Pc_r(\Y))$.
	Then
	\begin{equation}
		\label{eq:primal cvx hull no atoms}
		\inf_{\pi \in \Pi(\mu,\nu)} \int_{\X} C(x,\pi_x) \, \mu(dx) = \inf_{\pi \in \Pi(\mu,\nu)} \int_\X C^{\ast\ast}(x,\pi_x) \, \mu(dx).
	\end{equation}
	In particular, the minimization problems \eqref{eq:convexify wot'}-\eqref{eq:primal cvx hull} yield the same value.
\end{theorem}

\begin{proof}
	Since $\mu$ is continuous, by Proposition \ref{prop:couplings.dense} the values of \eqref{WOT} and \eqref{WOT'} coincide.
	Then, by Proposition \ref{prop:convexify} we have that the right-hand side of \eqref{eq:primal cvx hull no atoms} and \eqref{WOT'} yield the same value.
	Hence we conclude that there holds equality in \eqref{eq:primal cvx hull no atoms}.
	The last statement follows by Proposition \ref{prop:convexify}, since $C \ge \tilde C \ge C^{\ast\ast}$.
\end{proof}

\begin{proof}[Proof of Proposition~\ref{WTFTheorem}]
First, we note that \eqref{eq.rob} can be recast as a \eqref{WOT'}-problem.
Indeed, consider $C\colon \R \times \Pc_r(\R) \to \R$, $C(x,p) := \hat \Phi(x,b(p))$.
Then, we have by Proposition \ref{pr:ProcessInterpretation} that
\begin{equation}
	\label{eq:WTReformulation}
	\sup_{\Q \in \mathcal Q(\mu,\nu)} \E_\Q \left[ \hat \Phi(X_1,Y_1) \right] = \sup_{P \in \Lambda(\mu,\nu)} \int C(x,p) \, P(dx,dp) = -\inf_{P \in \Lambda(\mu,\nu)} \int -C(x,p) \, P(dx,dp),
\end{equation}
where in the last equality we switch from a $\sup$- to an $\inf$-formulation to match \eqref{WOT'}.
Since $\hat \Phi$, and thus $C$, is
continuous, we obtain from Proposition~\ref{prop:convexify} that
\[
	\inf_{P \in \Lambda(\mu,\nu)} \int -C(x,p) \, P(dx,dp) = \inf_{\pi \in \Pi(\mu,\nu)} \int_\X	(-C)^{\ast\ast}(x,\pi_x) \, \mu(dx).
\]
Finally, as $\mu$ is continuous, we can apply Theorem~\ref{thm:convexify} and obtain \eqref{WTFormulation}.
\end{proof}

\begin{remark} \label{rem:convexify}
	In this remark we show that under certain assumptions \eqref{WOT} for concave cost $C$ reduces to a classical optimal transport problem.
	Recall the definition of $\tilde C(x,\cdot)$ from Lemma \ref{lem:CastleqtildeC}.
	Assume that $C(x,\cdot)$ is concave in the following sense:
	\begin{equation}
		\label{eq:concave Jensen}
		\int_{\Pc_r(\Y)} C(x,q) \, Q(dq) \leq C(x,I(Q))\quad \forall Q \in \Pc_r(\Pc_r(\Y)),
	\end{equation}
	which is satisfied as soon as $C(x,\cdot)$ is concave, lower/upper
	semicontinuous and sufficiently bounded from above, by Lemma \ref{lem:Jensen}.
	It turns out that then $\tilde C(x,\cdot)$ is in fact linear. Indeed,
	let $p \in \Pc_r(\Y)$. The pushforward measure $Q^\ast := (y \mapsto \delta_y)_\ast p$ clearly constitutes an element of $\Pc_r(\Pc_r(\Y))$ such that $I(Q^\ast) = p$. Consequentially we deduce
	\begin{equation}
		\label{eq:concavity ineq}
		\tilde C(x,p) = \inf_{\substack{Q \in \Pc_r(\Pc_r(\Y)) \\ I(Q) = p}} \int_{\Pc_r(\Y)} C(x,q) \, Q(dq) \leq \int_{\Pc_r(\Y)} C(x,q) \, Q^\ast(dq) = \int_\Y C(x,\delta_y) \, p(dy).
	\end{equation}
	On the other hand, due to \eqref{eq:concave Jensen} we have
	\begin{equation}
		\label{eq:concavity ineq2}
		\int_\Y C(x,\delta_y) \, p(dy) = \int_{\Pc_r(\Y)} \int_\Y C(x,\delta_y) \, q(dy) \, Q(dq)
		\leq \int_{\Pc_r(\Y)} C(x,q) \, Q(dq),
	\end{equation}
	for all $Q\in\Pc_r(\Pc_r(\Y))$ with $I(Q) = p$. Thus the left-hand side of \eqref{eq:concavity ineq2}
	is dominated by $\tilde C(x,p)$. By combining this with \eqref{eq:concavity ineq}, we derive
	\[ \tilde C(x,p) = \int_\Y C(x,\delta_y) \, p(dy).	\]
	Therefore, \eqref{eq:convexify wot'} reduces to a classical optimal transport problem with cost 
	$c(x,y) := C(x,\delta_y)$:
	\begin{equation*}
	\inf_{P \in \Lambda(\mu,\nu)} \int_{\X \times \Pc_r(\Y)} C(x,p) \, P(dx,dp) = \inf_{\pi \in \Pi(\mu,\nu)} \int_{\X \times \Y} c(x,y) \, \pi(dx,dy).
	\end{equation*}
\end{remark}

\subsection{Barycentric cost functions} \label{ss:barycentric}
In this section we derive some properties of weak transport problem for `barycentric' cost functions and use them to prove the main results stated in the introduction.
We call a cost function $C \colon \X \times \Pc_1(\R^d) \to (-\infty,\infty]$ barycentric if, for any $x \in \X$, $p,q \in \Pc_r(\R^d)$, we have the implication
\begin{equation}\label{eq.bary}
	b(p) = b(q) \implies C(x,p) = C(x,q),
\end{equation}
where we recall that $b(p)$ denotes the barycenter of $p$.
Equivalently,  $C$ is barycentric if there exists $c \colon \X \times \R^d \to (-\infty,\infty]$ such that
\[	C(x,p) = c\left(x, b(p)\right), \quad (x,p) \in \X \times \Pc_1(\R^d).	\]

Weak optimal transport problems where the cost function $C$ is of barycentric-type have been recently investigated by various authors in different contexts, see \cite{GoRoSaSh18}, \cite{GoRoSaTe17}, \cite{GoJu18}, \cite{Sh16}, \cite{AlCoJo17b}, \cite{DaDeTz17}, \cite{BaBePa18}, \cite{BaBePa19}, \cite{BaPa20}.
Typically in these papers, $\X$ was also given as $\R^d$ and $C(x,p) = \theta(b(p)-x)$ for some convex $\theta \colon \R^d\to \R$.
Motivated by functional inequalities, this particular problem was explored by \cite{GoRoSaTe17}, \cite{GoRoSaSh18}, \cite{Sh16}, \cite{GoJu18}.
On the other hand, in \cite{FaGoPr20} it was used to give a new proof to Cafarelli's contraction theorem.
In \cite{AlCoJo17b} the authors are mainly motivated by applications in robust mathematical finance. 
They use the barycentric WOT problem to construct a sampling technique preserving the convex order, which can then be used to approximate martingale optimal transport problems.
Finally, in \cite{DaDeTz17} the barycentric WOT problem appears in the context of revenue maximization and mechanism design in economics.
For a clearer connection of this topic to WOT we refer to \cite{BaPa20}.

The next proposition fleshes out the intrinsic connection of \eqref{WOT'} with barycentric $C$ and the convex order (and thereby convex functions). 
Recall that we denote by $\leq_c$ the convex order between probabilities.
For the special case when the cost $C(x,p)$ is given by $\theta(b(p)-x)$ the results of the next proposition can be (partially) found in \cite{GoRoSaTe17} and \cite{BaBePa18}.

\begin{proposition} \label{prop:convex order barycentric C}
	Let $\mu \in \Pc_r(\X)$, $\nu \in \Pc_r(\R^d)$, and $C$ be measurable and barycentric such $\int C(x,p) \, P(dx,dp)$ is well-defined for any $P \in \Lambda(\mu,\nu)$. Then
\begin{enumerate}[label = (\alph*)]
\item \label{it:convex order barycentric C 1}
$\displaystyle{\inf_{P \in \Lambda(\mu,\nu)} \int_{\X \times \Pc_r(\R^d)} C(x,p) \, P(dx,dp) = \inf_{\nu' \leq_c \nu} \inf_{\pi \in \Pi(\mu,\nu')} \int_{\X \times \R^d} C(x,\delta_y) \, \pi(dx,dy)}$.\\
\item \label{it:convex order barycentric C 2} In addition, if $C$ is lower semicontinuous and satisfies \eqref{eq:lower bound}, then 
\begin{align*}
	\inf_{P \in \Lambda(\mu,\nu)} & \int_{\X \times \Pc_r(\R^d)} C(x,p) \, P(dx,dp) 
\\ &= \sup \left\{ \nu(\psi) + \mu(R_C\psi) \colon \psi \in \mathcal C_r(\R^d) \text{ and concave} \right\}
\\	&= \sup \left\{ \mu(\phi) + \nu(\psi) \colon \phi \in \mathcal C_r(\X), \text{concave } \psi \in \mathcal C_r(\R^d), \phi(x) + \psi(y) \leq C(x,\delta_y)~\forall (x,y) \in \X \times \R^d\right\}
\\	&= \sup\left\{ \mu(\phi) + \nu(\psi) \colon \phi \in \mathcal C_r(\X), \psi \in \mathcal C_r(\R^d), \Delta \colon \R^d \to \R^d, \right. 
\\ & \hspace{3cm} \left. \phi(x) + \psi(z) + \Delta(y)(z - y) \leq C(x,\delta_y) ~\forall (x,y,z) \in \X \times \R^d \times \R^d \right\},
\end{align*}
where $R_C\psi$ is defined as in Theorem \ref{thm:duality}.
	\end{enumerate}	
\end{proposition}

\begin{proof}[Proof of Proposition~\ref{prop:convex order barycentric C}]
	
	We first show \ref{it:convex order barycentric C 1}.
	Let $\pi \in \Pi(\mu,\nu')$ with $\nu' \leq_c \nu$.
	By Strassen's theorem there is a martingale coupling $\pi^M \in \Pi_M(\nu',\nu)$, where 
	$\Pi_M(\nu',\nu) := \left\{ \pi' \in \Pi(\nu',\nu) \colon b(\pi'_x) = x~\mu\text{-a.s.} \right\}$.
	Define
	\[  P(dx,dp) := \mu(dx) \, \int_{\R^d} \delta_{\pi^M_{y}}(dp) \, \pi_x(dy). \]
	Evidently, the first marginal of $P$ is $\mu$. Then the next line of computations establishes $P \in \Lambda(\mu,\nu)$ thanks to \eqref{eq:def_Lambda2}:
	\[ \proj^2_\ast \left( \hat I(P) \right) =  \int_{\X \times \R^d} \pi^M_y \, \pi(dx,dy) = \int_{\R^d} \pi^M_y \, \nu'(dy) = \proj^2_\ast(\pi^M) = \nu. \]
	Since $C$ is barycentric, we find
	\begin{equation} \label{eq:bary1}
	\int_{\X \times \Pc_1(\R^d)} C(x,p) \, P(dx,dp) = \int_{\X \times \Pc_1(\R^d)} C(x, \pi^M_y) \, \pi(dx,dy) = \int_{\X \times \R^d} C(x,\delta_y) \, \pi(dx,dy).
	\end{equation}
	We readily derive from \eqref{eq:bary1} that the r.h.s.\ in item \ref{it:convex order barycentric C 1} dominates the l.h.s.

	To derive the reverse inequality, we note that any $P \in \Lambda(\mu,\nu)$ induces a measure $\pi \in \Pc(\X \times \R^d)$ with second marginal $\nu' \leq_c \nu$, given by $\pi := ((x,p) \mapsto (x,b(p)))_\ast P$, so that
	\begin{multline} \label{eq:bary2}
	\int_{\X \times \Pc_1(\R^d)} C(x,p) \, P(dx,dp) = \int_{\X \times \Pc_1(\R^d)} C\left(x,\delta_{b(p)}\right) \, P(dx,dp) \\ = \int_{\X \times \R^d} C(x,\delta_y) \, \pi(dx,dy).
	\end{multline}
	 To see that $\proj^2_\ast \pi = \nu'$ is in convex order with $\nu$, we pick any convex $\theta \colon \R^d \to \R$ and find that
	\[
		\nu'(\theta) = \int_{\X \times \Pc_1(\R^d)} \theta(b(p)) \, P(dx,dp) \leq \int_{\X \times \Pc_1(\R^d)}
		p(\theta) \, P(dx,dp) = \nu(\theta),
	\]
by Jensen's inequality,	which means that $\nu' \leq_c \nu$. Hence, we obtain by \eqref{eq:bary2} that the l.h.s.\ in item \ref{it:convex order barycentric C 1} dominates the r.h.s., and therefore we have equality.
	
	Now, we prove \ref{it:convex order barycentric C 2}.
	To see the first equality, note that by Theorem \ref{thm:duality} it suffices to verify that the supremum in the r.h.s.\ of \eqref{eq:duality 1} can be restricted to concave $\psi \in \mathcal C_r(\R^d)$.
	To this end, let $\psi \in \mathcal C_r(\R^d)$ and $\psi^{\ast\ast}$ be its convex envelope.
	For any $y \in \R^d$, we have by Jensen's inequality
	\[ \psi^{\ast\ast}(y) = \inf_{p \in \Pc_r(\R^d),~b(p) = y} p(\psi^{\ast\ast}) \leq \inf_{p \in \Pc_r(\R^d),~b(p) = y} p(\psi). \]
	On the other hand, the right-hand side is convex and dominated by $\psi$, whence the inequality is actually
	an equality. Therefore,	
	\begin{align}\nonumber
	R_C\psi(x) = \inf_{p \in \Pc_r(\R^d)} -p(\psi) + C(x,p) &= \inf_{y \in \R^d} C(x,\delta_y) + \inf_{p \in \Pc_r(\R^d),~b(p) = y} -p(\psi) \\ \label{eq:bary 3}
	&= \inf_{y \in \R^d} C(x,\delta_y) + (-\psi)^{\ast\ast}(y)= R_C (-(-\psi)^{\ast\ast})(x).
	\end{align}
	Since $-(-\psi)^{\ast\ast}$ is concave and dominates $\psi \in \mathcal C_r(\R^d)$, we either have that $-(-\psi)^{\ast\ast} \in \mathcal C_r(\R^d)$ or $-(-\psi)^{\ast\ast} = \infty$.
	If we are in the latter case, we have by \eqref{eq:bary 3} that $R_C\psi = -\infty$, thus, $\nu(\psi) + \mu(R_C\psi) = -\infty$.
	Hence, we may assume w.l.o.g.\ that $-(-\psi)^{\ast\ast} \in \mathcal C_r(\R^d)$ which yields
	\[
		\nu(\psi) + \mu(R_C\psi) \le \nu(-(-\psi)^{\ast\ast}) + \mu(-(-\psi)^{\ast\ast}),
	\]
	and conclude that we may restrict to concave functions in the r.h.s.\ of \eqref{eq:duality 1}.

 The second equality follows from the first equality with the same line of argument as in the proof of Theorem~\ref{thm:duality}.

	For the third equality, we will show that any admissible pair in the second supremum of item \ref{it:convex order barycentric C 2} admits a better admissible pair in the third supremum, and vice versa. 
	To this end, consider any pair $(\phi,\psi) \in \mathcal C_r(\X) \times \mathcal C_r(\R^d)$.
	If $\psi$ is concave, then there exists a measurable selection $\Delta \colon \R^d \to \R^d$ of the subgradient of $-\psi$, i.e.,
	\begin{equation} \label{eq:bary4}
		\phi(x) + \psi(z) + \Delta(y) (z - y) \leq \phi(x) + \psi(y)\quad \forall (x,y,z) \in \X \times \R^d \times \R^d.
	\end{equation}
	We derive from \eqref{eq:bary4} that $(\phi,\psi,\Delta)$ is admissible in the last supremum whenever $(\phi,\psi)$ is admissible in the second supremum.
	
	On the other hand, if there is a measurable $\Delta \colon \R^d \to \R^d$ with
	\[	\phi(x) + \psi(z) + \Delta(y)(z - y) \leq C(x,\delta_y)\quad \forall (x,y,z) \in \X \times \R^d \times \R^d,	\]
	we can define the concave function $\tilde \psi(z) := \inf\left\{ C(x,\delta_y) - \phi(x) - \Delta(y)(z-y) \colon (x,y)\in\X\times\R^d\right\}$. 
	Since $\tilde \psi$ is finitely valued, dominates $\psi$ and is concave, we find that $\tilde \psi \in \mathcal C_r(\R^d)$.
	In particular, $(\phi,\tilde \psi)$ is admissible for the second supremum
	which concludes the proof.
\end{proof}

\begin{proof}[Proof of Theorem \ref{SuperRepIntro} and Theorem \ref{BaryDuality}]
	Theorem~\ref{BaryDuality} follows directly from Theorem \ref{thm:convexify} and Proposition \ref{prop:convex order barycentric C}.

	As already noticed in the proof of Proposition \ref{WTFTheorem}, the robust optimization problem in \eqref{eq.rob} can be reformulated as in \eqref{eq:WTReformulation}.
	Finally Theorem~\ref{SuperRepIntro} is a consequence of Proposition~\ref{pr:ProcessInterpretation} and Proposition~\ref{prop:convex order barycentric C}.
\end{proof}

\subsection{Primal and dual optimizers for convex barycentric costs}\label{app.dualopt}

In this section we characterize primal and dual optimizers in the case of vanilla derivatives as announced in Section \ref{sect.caplets}.
Hence we consider costs of the form 
\begin{equation}\label{eq.cost.cb}
 C(x,p) = \theta(b(p)-x),
\end{equation}
with $\theta \colon \R \to \R$ convex, and investigate primal and dual optimizers for the problems 
\begin{align}\label{eq.inf}
	\inf_{\pi \in \Pi(\mu,\nu)} \int_\R \theta(b(\pi_x)-x) \, \mu(dx),\\ \label{eq.sup}
	\sup_{\pi \in \Pi(\mu,\nu)} \int_\R \theta(b(\pi_x)-x) \, \mu(dx),
\end{align}
where $\mu,\nu \in \Pc_1(\R)$.

\begin{corollary}
	\label{cor:dual optimizers}
	The optimization problems \eqref{eq.inf} and \eqref{eq.sup} satisfy
	\begin{gather}
		\label{eq:1d-lower bound}
\inf_{\pi \in \Pi(\mu,\nu)} \int \theta\left(b(\pi_x)-x\right) \, \mu(dx) = \sup_{\psi \in \mathcal C_1(\R), \, concave} \nu(\psi) + \mu(R_C\psi), \\
		\label{eq:1d-upper bound}
		\sup_{\pi \in \Pi(\mu,\nu)} \int_\R \theta(b(\pi_x)-x) \, \mu(dx) = \sup_{\pi \in \Pi(\mu,\nu)} \int \theta(y-x) \, \pi(dx,dy) = \inf_{\psi \in \mathcal C_1(\R), \, convex} \nu(\psi) + \mu(\overline{R}_C \psi),
	\end{gather}
	where $\overline{R}_C \psi(x) := \sup_{y \in \R} \theta(y-x)-\psi(y)$ and $R_C\psi(x) = \inf_{y \in \R} \theta(y - x) - \psi(y)$.
\end{corollary}

\begin{proof}
	For $C(x,p) = \theta(x - b(p))$ where $\theta$ is convex, we have by Theorem 2.9 in \cite{BaBePa18} that \eqref{WOT} = \eqref{WOT'}.
	Hence, \eqref{eq:1d-lower bound} is readily deduced from item \ref{it:convex order barycentric C 2} of Proposition \ref{prop:convex order barycentric C}.

	To see \eqref{eq:1d-upper bound}, we note that by Remark~\ref{rem:convexify}
	\begin{equation}
		\label{eq:sup OT problem}
		\sup_{\pi \in \Pi(\mu,\nu)} \int_\R \theta(b(\pi_x)-x) \, \mu(dx)=\sup_{\pi \in \Pi(\mu,\nu)} \int_\R \theta(y-x) \, \pi(dx,dy),
	\end{equation}
	that is, the problem reduces to a classical optimal transport problem.
	Then \eqref{eq:1d-upper bound} follows again from item \ref{it:convex order barycentric C 2} of Proposition \ref{prop:convex order barycentric C}.
\end{proof}

Let us now consider primal optimizers of \eqref{eq.inf} and \eqref{eq.sup}.
By Proposition \ref{prop:convex order barycentric C} we have
\begin{equation}
	\label{eq:1d-lower bound'}
\inf_{\pi \in \Pi(\mu,\nu)} \int_\R \theta(b(\pi_x)-x) \, \mu(dx) = \inf_{\eta \leq_c \nu}
	\inf_{\pi \in \Pi(\mu,\eta)} \int_\R \theta(y-x) \, \pi(dx,dy).
\end{equation}
A minimal measure on $\R$ for the right-hand side, denoted by $\nu^\ast \leq_c \nu$, is given by the 
image of $\mu$ under the weak monotone rearrangement $\underline T$ of $\mu$ and $\nu$, see \cite{BaBePa19}.
Note that 
\[ \underline T = F_{\nu^\ast}^{-1} \circ F_\mu\quad\mu\text{-a.s.},\]
where $F_p$ and $F_p^{-1}$ denote the cumulative distribution function and the generalized inverse distribution function, resp.\ of the probability $p \in \Pc(\R)$.

\begin{proposition}
	\label{prop:primal optimizers}
	A primal optimizer of \eqref{eq.inf} is given by the coupling $\mu(dx)\, \pi^M_{\underline T(x)}(dy)$ where $\pi^M \in \Pi_M(\mu,\underline T(\mu))$ and $\underline T$ is the weak monotone rearrangement of $\mu$ and $\nu$.
	A primal optimizer of \eqref{eq.sup} is given by the anticomonotone coupling.
\end{proposition}

\begin{proof}
	The first statement can be found in Theorem 3.1 of \cite{BaBePa19} and the second in Theorem 3.1.2 of \cite{RaRu98}.
\end{proof}

Before turning to the description of dual optimizers of \eqref{eq.inf} and \eqref{eq.sup}, we define potential candidates in the subsequent lemma.
We remark that the construction in Lemma \ref{lem:dual optimizers} is not novel, but is included for the sake of completeness as the authors are unaware of a fitting reference.
A dual optimizer of the corresponding weak transport problem is constructed in \cite{Sh16a} under the assumption that $T$ is strictly increasing.
The author remarks that the assumption is simply to avoid technicalities.

\begin{lemma} \label{lem:dual optimizers}
	Let $\theta \colon \R \to \R$ be convex, $T\colon I \to \R$ a map where $I\subset \R$, $[a,b] = \overline{\text{co}(T(I))}$, and $y_0 \in (a,b)$.
	\begin{enumerate}[label = (\alph*)]
		\item \label{it:T inc} Assume that $T$ is nondecreasing and define
		\begin{gather} \label{def:S when T inc}
		\underline S(y) :=  \inf \{ x \in I \colon T(x) \geq y \}, \quad y \in (a,b), \\
		\label{def:psi when T inc}
		\underline \psi(y) := \int_{y_0}^y \partial_- \theta\left(z - \underline S(z)\right) \, dz, \quad y \in (a,b).
		\end{gather}
		Then $\underline \psi$ is continuous on $(a,b)$ and satisfies for all $x \in I$, $y' \in [T(x-),T(x)] \cap [a,b]$ that
		\begin{equation} \label{eq:psi opt when T inc}
		- \underline \psi(y') + \theta(y' - x) = \inf_{y \in [a,b]} - \underline \psi(y) + \theta(y - x).
		\end{equation}
		\item \label{it:T dec} Assume that $T$ is nonincreasing and define
		\begin{gather} \label{def:S when T dec}
		\overline S(y):= \sup \{ x \in I \colon T(x) \geq y \}, \quad y \in (a,b), \\
		\label{def:psi when T dec}
		\overline \psi(y) := \int_{y_0}^y \partial_- \theta\left(z - \overline S(z) \right) \, dz, \quad y \in (a,b).
		\end{gather}
		Then $\overline \psi$ is continuous on $(a,b)$ and satisfies for all $x \in I$, $y' \in [T(x+), T(x)] \cap [a,b]$ that
		\begin{equation} \label{eq:psi opt when T dec}
		- \overline \psi(y') + \theta(y' - x) = \sup_{y \in [a,b]} \overline \psi(y) + \theta(y - x).
		\end{equation}
	\end{enumerate}
\end{lemma}

\begin{proof}
	Since $\theta$ is convex, it is locally absolutely continuous, thus
	\[  \theta(y') - \theta(y) = \int_y^{y'} \partial_- \theta(z) \, dz, \]
	where we write $\partial_- \theta$ for the left derivative of $\theta$, which is nondecreasing.
	For fixed $x \in I$ we introduce
	\[	\underline f(z) := -\underline \psi(z) + \theta(z - x) \text{ and } \overline f(z) := -\overline \psi(z) + \theta(z-x),\]
	where $z \in (a,b)$. 
	Using that convex functions are almost surely differentiable, we have $dz$-almost surely on $(a,b)$
	\begin{align}
		\label{eq:differentials}
		\begin{split}
		\frac{d}{dz} \underline f(z) = - \partial_-\theta(z - \underline S(z)) + \partial_- \theta(z - x)\\
		\frac{d}{dz} \overline f(z) = -\partial_-\theta(z - \overline S(z)) + \partial_- \theta(z - x).
		\end{split}
	\end{align}
	The maps $\underline S$ and $\overline S$ are nonincreasing and nondecreasing, respectively.
	If $T$ is monotone, then for any $z \in (a,b)$ there are $\underline x,\overline x \in I$ such that
	\begin{align*}
		T \text{ nondecreasing} \colon\; 	x \in (-\infty,\underline x] \cap I \implies T(x) < z;\;\; \text{ and }\; 
		x \in [\overline x,\infty) \cap I \implies T(x) > z;
	\\	T \text{ nonincreasing} \colon\; x \in (-\infty, \underline x] \cap I \implies T(x) > z;\;\; \text{ and }\;
		x \in [\overline x,\infty) \cap I \implies T(x) < z.
	\end{align*}
	Hence, $\underline S$ and $\overline S$ are real-valued, which also shows that
	$\underline \psi$ and $\overline \psi$ are well-defined and continuous on $(a,b)$.
	From now on, we implicitly assume that $T$ is nondecreasing when talking about $\underline f$ and $\underline S$,
	whereas we assume that $T$ is nonincreasing when talking about $\overline f$ and $\overline S$.
	We have
	\begin{align*}
		z \in (T(x-),T(x)) \implies x \leq \underline S(z) \leq \underline S(T(x)) \leq x \iff x = \underline S(z),	\\	
		z \in (T(x+),T(x)) \implies x \geq \overline S(z) \geq \overline S(T(x)) \geq x \iff x = \overline S(z).
	\end{align*}
	Therefore, we obtain
	\begin{align*}
		\frac{d}{dz}\underline f \vert_{(T(x-),T(x))} = 0 \implies \underline f\vert_{[T(x-),T(x)]} = \underline f(T(x)),
	\\	\frac{d}{dz}\overline f \vert_{(T(x+),T(x))} = 0 \implies \overline f\vert_{[T(x+),T(x)]} = \overline f(T(x)).
	\end{align*}
	It is now sufficient to show that
	\begin{equation}
		\label{eq:differential ineq}
		\frac{d}{dz} \underline f \vert_{(a,T(x)]} \leq 0 \text{ and } \frac{d}{dz} \underline f \vert_{[(T(x),b)} \geq 0;
	\quad\frac{d}{dz} \overline f \vert_{(a,T(x)]} \geq 0 \text{ and } \frac{d}{dz}
		\overline f \vert_{[T(x),b)} \leq 0.
	\end{equation}
	Due to monotonicity and by definition of $\underline S$ and $\overline S$, we find:
	\begin{align}
		\label{eq:differential props}
		\begin{split}
		z \in (a,T(x)) \implies 
		\begin{cases} 
			\underline S(z) \leq \underline S(T(x)) \leq x \implies & z - \underline S(z) \geq z - x, 
		\\ x \leq \overline S(z) \phantom{\hspace{1.675cm}} \implies & z - \overline S(z) \leq z - x,
		\end{cases}
	\\	z \in (T(x), b) \implies 
		\begin{cases} 
			x \leq \underline S(z) \phantom{\hspace{1.671cm}} \implies & z - x \geq z - \underline S(z),
		\\	\overline S(z) \leq \overline S(T(x)) \leq x \implies & z - x \leq z - \overline S(z).
		\end{cases}
		\end{split}
	\end{align}
	As the left derivative of a convex function is nondecreasing, \eqref{eq:differential ineq} follows from \eqref{eq:differentials} and \eqref{eq:differential props}.
\end{proof}

In what follows we provide a (semi-)explicit representation of the dual optimisers.
\begin{proposition}
	\label{prop:dual optimizers}
	Assume there exist $f \in L^1(\mu)$, $g \in L^1(\nu)$ such that $\theta(y-x) \leq f(x) + g(y)$.
	Then the right-hand side of \eqref{eq:1d-lower bound} is attained by
	$\underline \psi$, given in \eqref{def:psi when T inc}, for $T = \underline T$ the weak monotone rearrangement of $\mu$ and $\nu$.

	Assume there exist $\tilde f \in L^1(\mu)$, $\tilde g \in L^1(\nu)$ such that $\theta(y-x) \geq \tilde f(x) + \tilde g(y)$.
	Then the right-hand side of \eqref{eq:1d-upper bound} is attained by $\overline \psi$, given in \eqref{def:psi when T dec},
	for $T(x) = \overline T(x) := F_\nu^{-1}(1 - F_\mu(x))$.
\end{proposition}

\begin{proof}
	Let $\pi^\ast$ be a minimizer of \eqref{eq.inf} given as in Proposition \ref{prop:primal optimizers} by
	\[	\pi^\ast(dx,dy) := \mu(dx) \, \pi_{\underline T(x)}^M(dy),\]
	where $\pi^M$ is an arbitrary martingale coupling
	with marginals $\mu$ and $\nu^\ast$, and $\underline T$ is nondecreasing and 1-Lipschitz.
	From \cite[Theorem 1.3]{BaBePa19} we have that
	\begin{equation}
		\label{eq:linear on irred comp}
		\pi^M_{\underline T(x)}\left( \left\{ y \in \R \colon \underline T(x) - x = \underline T(y) - y \right\} \right) = 1\quad \mu\text{-almost surely}.
	\end{equation}
	Since $\underline T$ is nondecreasing and 1-Lipschitz, the map $x \mapsto \underline T(x) - x$ is nonincreasing and
	\[ \left\{ y \in \R \colon \underline T(x) - x = \underline T(y) - y \right\} \]
	is a closed interval for every $x \in \R$.
	For $z,z'$ in the interior of $\underline T(\R)$ there are minimal $x,x' \in \R$ with $\underline T(x) = z$ and
	$\underline T(x') = z'$. Therefore, $\underline S(z) = x$, $\underline S(z') = x'$ (where $\underline S$ is as in Lemma \ref{lem:dual optimizers}), and
	\[	z < z' \implies z - \underline S(z) = \underline T(x) - x \geq \underline T(x') - x' = z' - \underline S(z'),	\]
	i.e. $z \mapsto z - \underline S(z)$ is nonincreasing, from which we deduce concavity of $\underline \psi$.
	Furthermore, $\underline \psi$ restricted to $\{y \in \R \colon \underline T(y) - y = \underline T(x) - x\}$ is linear,
	whereby we find by \eqref{eq:linear on irred comp} that $\mu$-almost surely 
	\[
		\int_\R \underline \psi(y) \, \pi^M_{\underline T(x)}(dy) = \underline \psi\left(b(\pi^M_{\underline T(x)})\right) = \underline \psi(\underline T(x)).
	\]
	Therefore, by Lemma \ref{lem:dual optimizers} \ref{it:T inc},
	\begin{equation} \label{eq:cons of linear}
		\int_\R \underline \psi(y) \, \pi^M_{\underline T(x)}(dy)  + R_C \underline \psi(x) = \theta(\underline T(x) - x).
	\end{equation}
	If $\theta(y - x) \leq f(x) + g(y)$ for $f \in L^1(\mu)$ and $g \in L^1(\nu)$, then the $\mu$-integral
	of $R_C\underline \psi$ and the $\nu$-integral of $\underline \psi$ are well-defined in $[-\infty,\infty)$.
	Due to concavity of $\underline \psi$ and as $\nu^\ast$ is in convex order smaller than $\nu$, we observe that the $\nu^\ast$-integral of $\underline\psi$ is well-defined in $[-\infty, \infty)$.
	The function $R_C \underline \psi$ is convex (as the infimum over a jointly
	convex function) yielding $R_C\underline \psi \in L^1(\mu)$. Therefore, by \eqref{eq:cons of linear}
	\begin{align*}
		\int_{\R \times\R} \theta(y - x) \, \pi^\ast(dx,dy) &= \int_\R \theta(\underline T(x) - x) \, \mu(dx) = \int_\R \underline \psi(\underline T(x)) + R_C\underline \psi(x) \, \mu(dx)
	\\	&= \int_\R \int_\R \psi(y) \, \pi^M_{\underline T(x)}(dy)  + R_C\underline \psi(x) \, \mu(dx) = \nu(\underline \psi) + \mu(R_C \underline \psi),
	\end{align*}
	hence $\underline\psi$ is a dual optimizer.

	To see the second assertion, note that the primal optimizer of the upper bound, \eqref{eq.sup}, is given
	by the anticomonotone coupling $\pi^{a} \in \Pi(\mu,\nu)$.
	We have
	\[	\pi_x^a\left([\overline T(x+),\overline T(x)]\right) = 1 \quad \mu\text{-almost surely.}	\]
	Recalling the definition of $\overline R_C\psi$ given in Corollary \ref{cor:dual optimizers}, we find
	$\overline \psi(y) + \overline R_C\psi(y) = \theta(y-x)$ $\pi^a$-almost surely.
	Moreover, $\overline \psi$ is convex on its domain, since $z \mapsto z - \overline S(z)$ (where $\overline S$ is as in Lemma \ref{lem:dual optimizers}) and
	$\partial_- \theta(z)$ are nondecreasing.
	If $\theta(y-x) \geq \tilde f(x) + \tilde g(y)$ for $\tilde f \in L^1(\mu)$ and $\tilde g \in L^1(\nu)$, then the integral of $\overline \psi$ w.r.t.\ $\nu$ and the integral of $\overline R_C\psi$
	w.r.t.\ $\mu$ are well-defined in $(-\infty, \infty]$.
	Therefore, $\overline \psi$ is a dual optimizer, since we find by Lemma \ref{lem:dual optimizers} \ref{it:T dec} that
	\[	\int_{\R \times \R} \theta(y-x) \, \pi^a(dx,dy) = \int_{\R \times \R} \overline \psi(y)
	+ \overline R_C \psi(x) \, \pi^a(dx,dy) = \nu(\overline \psi) + \mu(\overline R_C\psi).	\qedhere\]
\end{proof}

\begin{example}\label{exa.call}
	In the setting of Corollary \ref{cor:dual optimizers}, let $\theta$ be given by $y \mapsto y^+$, and assume that $\mu$ and $\nu$ are equivalent to the Lebesgue measure restricted to some interval $I$.
	Then the left derivative of $\theta$ satisfies
	\[	\partial_-\theta(y - x) = \begin{cases} 0 & y \leq x, \\ 1 & \text{else.}	\end{cases}	\]
	We have already seen in the proof of Corollary \ref{cor:dual optimizers} that 
	$y \mapsto y - \underline S(y)$ and $y \mapsto y - \overline S(y)$ are nonincreasing and nondecreasing,
	respectively. Therefore, there are uniquely determined points $\underline a,\overline a \in [-\infty,\infty]$
	such that
	\[
		y - \underline S(y) > 0 \quad \forall 	y < \underline a\quad\text{ and }\quad y - \overline S(y) > 0 \quad \forall y > \overline a.
	\]
	If $\underline a \in \R$, then
	\begin{align*}
		\underline \psi(y) = \int_{y_0}^y 1_{(-\infty,\underline a)}(z) \, dz = (\underline a - y_0)^+ - (\underline a - y)^+,
	\\	R_{C}(x) = \inf_{y \in \R} (\underline a - y)^+ + (y - x)^+ - (\underline a - y_0)^+ = (\underline a - x)^+ - (\underline a - y_0)^+,
	\end{align*}
	whereas if $\overline a \in \R$, then
	\begin{align*}
		\overline \psi(y) = \int_{y_0}^y 1_{(\overline a, \infty)}(z) \, dz = (y - \overline a)^+ - (y_0 - \overline a)^+,
	\\	\overline R_{C}(x) = \sup_{y \in \R} -(y - \overline a)^+ + (y_0 - \overline a)^+ + (y -x)^+ = -(x - \overline a)^+ + (y_0 - \overline a)^+.
	\end{align*}
	Note that the constants are canceling out.
	Therefore, removing the constants and differentiating yields the triplets
	\begin{align}\label{eq.trip.1}
		\left(\underline \phi(x_1),\underline \psi(y_2),\underline \Delta(y_1)\right) := \left( (\underline a - x_1)^+ , -(\underline a - y_2)^+, -\partial_-\theta(\underline a - y_1) \right),\\		
	\left(\overline \phi(x_1), \overline \psi(y_2), \overline \Delta(y_1)\right) := \Big( -(x_1 - \overline a)^+ , (y_2 - \overline a)^+, \partial_-\theta(y_1- \overline a) \Big),\label{eq.trip.2}
	\end{align}
	for $(x_1,y_1,y_2) \in \R^3$, which are optimizers of
	\begin{align*}
		\sup\left\{ \mu(\phi) + \nu(\psi) \colon \phi(x_1) + \psi(y_2) + \Delta(y_1)(y_2 - y_1) \leq (y_1 - x_1)^+\text{ for all }x_1,y_1,y_2 \right\},
	\\	\inf\left\{ \mu(\phi) + \nu(\psi) \colon \phi(x_1) + \psi(y_2) + \Delta(y_1)(y_2 - y_1) \geq (y_1 - x_1)^+\text{ for all }x_1,y_1,y_2 \right\},
	\end{align*}
	respectively.
\end{example}

\section*{Acknowledgments}
MB acknowledges support from FWF through grant no.\ Y00782.

\vspace{0.7cm}

\noindent No data was used for this research.

\bibliography{joint_biblio}
\bibliographystyle{apalike}
\end{document}